\documentclass[10pt]{article}
\RequirePackage{amsthm,amsmath,amsfonts,amssymb}
\RequirePackage[numbers]{natbib}
\RequirePackage[colorlinks,citecolor=blue,urlcolor=blue]{hyperref}

\usepackage{graphicx}

\usepackage[mathscr]{eucal}
\usepackage{enumerate}

\usepackage{cases}

\usepackage{diagbox}

\numberwithin{equation}{section}
\theoremstyle{plain}
\newtheorem{thm}{Theorem}[section]

\newtheorem{lem}{Lemma}[section]
\newtheorem{prop}{Proposition}[section]
\newtheorem{rem}{Remark}[section]

\usepackage{color}
\newcommand{\abs}[1]{\left| #1\right|}

\newcommand{\E}{\mathbb{E}}
\newcommand{\PP}{\mathbb{P}}
\newcommand{\R}{\mathbb{R}}

\begin{document}


\title{Hellinger and total variation distance in approximating L\'evy driven SDEs}

\author{
 Emmanuelle Cl\'ement
 \thanks{LAMA, Univ Gustave Eiffel, Univ Paris Est Creteil, CNRS, F-77447 Marne-la-Vall\'ee, France. }}


\date{ Revision : 02/03/22}
\maketitle

\begin{quote}
\noindent
\textbf{Abstract.} In this paper, we get some convergence rates in total variation distance in approximating discretized paths of L\'evy driven stochastic differential equations, assuming that the driving process is locally stable. The particular case of the  Euler approximation is studied. Our results are based on sharp local estimates in Hellinger distance obtained using Malliavin calculus for jump processes. 

\noindent
\textbf{MSC $2020$}.  60H10,  60G51, 60B10, 60H07.

\noindent
\textbf{Key words}: total variation, Hellinger Distance, L\'evy Process, Stable Process, Stochastic Differential Equation.
\end{quote}

\section{Introduction}

On a complete probability space $( \Omega, \mathcal{F}, \mathbb{P})$, we consider the process $(X_t)_{t \in [0,1]}$ solution of the stochastic equation
\begin{equation} \label{E:EDS}
X_t=x_0+ \int_0^t b(X_s) ds+ \int_0^t a(X_{s-}) d L_s, 
\end{equation}
where $L$ is a pure jump locally stable L\'evy process. Pure jump driven stochastic equations  are widely used to model dynamic phenomena appearing in many fields such as insurance and finance and approximation of such processes attracts many challenging problems. A large part of the literature is devoted to the study of weak convergence at terminal date  $\mathbb{E}g(X_T)- \mathbb{E}g(\overline{X}_T)$ (we assume in this paper that $T=1$), where $\overline{X}$ is a numerical scheme. Let us mention some results obtained in approximating L\'evy driven stochastic equations by the simplest and widely used Euler scheme. The weak order $1$ for equations with smooth coefficients and for smooth functions $g$ is obtained in Protter and Talay \cite{PT97} and some extensions to H\"{o}lder coefficients are studied in  Mikulevi\v{c}ius and Zhang \cite{Miku11} and Mikulevi\v{c}ius \cite{Miku12}. Expansions of the density are considered in Konakov and Menozzi \cite{Konakov11}. Turning to pathwise approximation, convergence rates in law for the error process are obtained by Jacod \cite{Jac2004} and
some strong convergence results have been established in Mikulevi\v{c}ius and Xu \cite{Miku18}. To overcome the difficulties  related to the simulation of the small jumps of $L$, more sophisticated schemes have been considered. We quote among others the works of Rubenthaler \cite{Rub03} and Kohatsu-Higa and  Tankov \cite{Tankov10}.

In this paper, we consider a different control of the accuracy of approximation and we focus on high-frequency pathwise approximation of \eqref{E:EDS} in total variation distance. We mention that this result has also additional interesting consequence in parametric statistics and permits to derive asymptotic properties such as efficiency for the statistical experiment based on high-frequency observation of the stochastic equation by using the numerical scheme for which the log-likelihood function is explicitly connected to the distribution of the driving L\'evy process. 

 We now precise the schemes considered in the present work. To deal with small values of the Blumenthal-Getoor index of $L$ (characterizing the jump activity), we not only consider the Euler approximation of \eqref{E:EDS} but also a  scheme with better drift approximation.
Introducing  the time discretization $(t_i)_{0 \leq i \leq n}$ with $t_i=i/n$, we approximate the process $(X_t)_{t \in [0,1]}$ by $(\overline{X}_t)_{t \in [0,1]}$ defined by $\overline{X}_0=x_0$ and for $t \in [t_{i-1},t_i]$, $1 \leq i \leq n$
\begin{equation} \label{E:Euler}
\overline{X}_t= \xi_{t-t_{i-1}}(\overline{X}_{t_{i-1}}) + a(\overline{X}_{t_{i-1}}) (L_t-L_{t_{i-1}}), 
\end{equation}
where $(\xi_t(x))_{t \geq 0}$ solves the ordinary equation 
\begin{equation} \label{E:EDO}
\xi_t(x)=x+ \int_0^t b(\xi_s(x)) ds. 
\end{equation}
Approximating $\xi$ by
\begin{equation*} \label{E:EulerEDO}
\tilde{\xi}_t(x)=x+b(x) t,
\end{equation*}
we obtain the Euler approximation $(\tilde{X}_t)_{t \in [0,1]}$ defined by $\tilde{X}_0=x_0$ and for $t \in [t_{i-1},t_i]$, $1 \leq i \leq n$
\begin{equation} \label{E:TrueEuler}
\tilde{X}_t= \tilde{X}_{t_{i-1}}+b( \tilde{X}_{t_{i-1}}) (t-t_{i-1}) + a(\tilde{X}_{t_{i-1}}) (L_t-L_{t_{i-1}}).
\end{equation}
Our aim is to study the rate of convergence of $(\overline{X}_{t_i})_{0 \leq i \leq n}$ or $(\tilde{X}_{t_i})_{0 \leq i \leq n}$ to $(X_{t_i})_{0 \leq i \leq n}$ in total variation distance. Let us present briefly our results. For the scheme \eqref{E:Euler}, we obtain some rates of convergence, depending on the jump activity index $\alpha \in (0,2)$. Essentially the rate of convergence is of order $1/ n^{1/ \alpha-1/2}$ if $\alpha >1$ and $1/n^{1/2 - \varepsilon}$ if $\alpha \leq 1$. If the scale coefficient $a$ is constant, we obtain in some cases  the rate $1/\sqrt{n}$ for any value of $\alpha$. For the Euler scheme, the results are similar if $\alpha \geq 1$ but are working less well if $\alpha<1$, and we have no rate at all if $\alpha \leq 2/3$. Intuitively, on a time step, the drift term has order $1/n$ and the stochastic jump part has order $1/n^{1/ \alpha}$, consequently if the jump activity is small the main part of the stochastic equation is the drift and an approximation of  \eqref{E:EDO} with higher order than the Euler one is required. 

To get these results, our methodology consists in estimating the local Hellinger distance at time $1/n$ and to conclude by tensorisation. Using Malliavin calculus for jump processes, we can bound the Hellinger distance by the $L^2$-norm of a Malliavin weight. The difficult part is next to identify a sharp rate of convergence for this weight. This is done by remarking some judicious compensations between the rescaled jumps.

The paper is organized as follows. Section \ref{S:Not} introduces the notation and some preliminary results. Bounds for the local Hellinger distance are given in Section \ref{S:RateH}. The main results are presented in Section \ref{S:TV}. They concern the pathwise approximation in total variation distance and also include one step approximation results in small time. The optimality of the local bounds is discussed on some specific examples.    Section \ref{S:Malliavin} contains the technical part of the paper involving Malliavin calculus and the proof of the local estimates of  Section \ref{S:RateH}.

\section{Preliminary results and notation} \label{S:Not}

We first recall some properties of total variation and Hellinger distance (see Strasser \cite{Strasser}). Let $P$ and $Q$ be two probability measures on $(\Omega, \mathcal{A})$ dominated by $\nu$, 
 the total variation distance between $P$ and $Q$  on $(\Omega, \mathcal{A})$ is defined by
\begin{equation*} 
d_{TV}(P,Q)= \sup_{A \in \mathcal{A}} | P(A)-Q(A)|= \frac{1}{2} \int \left| \frac{dP}{d \nu}- \frac{d Q}{ d \nu} \right| d \nu .
\end{equation*}
The total variation distance can be estimated by using the Hellinger distance $H(P,Q)$ defined by
\begin{equation} \label{E:DHell}
H^2(P,Q)=  \int \left( \sqrt{ \frac{dP}{d \nu}}- \sqrt{\frac{d Q}{ d \nu}}\right)^2 d \nu =2\left(1- \int \sqrt{ \frac{dP}{d \nu}} \sqrt{\frac{d Q}{ d \nu}} d \nu\right)
\end{equation}
and we have
$$
\frac{1}{2}H^2(P,Q) \leq d_{TV}(P,Q) \leq H(P,Q).
$$
If $P$, respectively $Q$, is the distribution of a random variable $X$, respectively $Y$, we also use the notation $d_{TV}(X,Y)$ for $d_{TV}(P,Q)$ and $H(X,Y)$ for $H(P,Q)$. The Hellinger distance has interesting properties, in particular for product measures 
$$
H^2(\otimes_{i=1}^n P_i, \otimes_{i=1}^n Q_i)\leq \sum_{i=1}^n H^2(P_i,Q_i).
$$
We extend this property in the next proposition to the distribution of Markov chains. 

Let $(X_i)_{i \geq 0}$ and $(Y_i)_{i \geq 0}$ be two homogenous Markov chains on $\R$ with transition density $p$ and $q$ with respect to the Lebesgue measure. We define the conditional Hellinger distance between $X_1$ and $Y_1$ given $X_0=Y_0=x$ by
\begin{equation*}
H^2_x(p,q)= \int \left(\sqrt{p(x,y)}-\sqrt{q(x,y)} \right)^2 dy.
\end{equation*}
We denote by $P^n$, respectively $Q^n$, the distribution of $(X_i)_{1 \leq i \leq n}$ given $X_0=x_0$, respectively $(Y_i)_{1 \leq i \leq n}$ given $Y_0=x_0$ (the two Markov chains have the same initial value), then we can bound $H(P^n, Q^n)$ with $H_x(p,q)$.

\begin{prop}\label{P:Tenso}
With the previous notation, we have
$$
H^2(P^n, Q^n) \leq \frac{1}{2}\sum_{i=1}^n \left( \E H^2_{X_{i-1}}(p,q) +\E H^2_{Y_{i-1}}(p,q) \right) \leq n \sup_{x \in \R} H_x^2(p,q).
$$
\end{prop}

\begin{proof}
We have from \eqref{E:DHell}
\begin{eqnarray*}
H^2(P^n, Q^n)   =  2\left( 1- \int_{\R^n} \left( \prod_{i=1}^n p(x_{i-1}, x_{i}) \prod_{i=1}^n q(x_{i-1}, x_{i})\right)^{1/2} d x_1 \ldots d x_n \right). 
 \end{eqnarray*}
 But
 \begin{eqnarray*}
\int_{\R} \sqrt{p(x_{n-1}, x_{n})  q(x_{n-1}, x_{n})} dx_n = 1-\frac{1}{2} H^2_{x_{n-1}}(p,q),
 \end{eqnarray*}
 consequently
 \begin{eqnarray*}
 H^2(P^n, Q^n)   =   H^2(P^{n-1}, Q^{n-1})  \hspace{6cm}\\
 + \int_{\R^{n-1}} \left( \prod_{i=1}^{n -1}p(x_{i-1}, x_{i}) \prod_{i=1}^{n-1} q(x_{i-1}, x_{i})\right)^{1/2} H^2_{x_{n-1}}(p,q)d x_1 \ldots d x_{n-1}, 
 \end{eqnarray*}
 and from the inequality $\sqrt{ab} \leq \frac{1}{2}(a+b)$, this gives
 \begin{eqnarray*}
H^2(P^n, Q^n) &   \leq &
  H^2(P^{n-1}, Q^{n-1}) + \frac{1}{2} ( \E H^2_{X_{n-1}}(p,q) + \E H^2_{Y_{n-1}}(p,q)).
\end{eqnarray*}
We deduce then the first inequality in Proposition \ref{P:Tenso} by induction, the second inequality is immediate.
\end{proof}
The result of Proposition \ref{P:Tenso} motivates the study of the Hellinger distance between $X_{1/n}$ and $\overline{X}_{1/n}$ given $X_0=\overline{X}_0=x$ (respectively $\tilde{X}_{1/n}$) to bound $d_{TV} ((X_{i/n})_{  i }, (\overline{X}_{i/n})_{  i })$ (respectively $d_{TV} ((X_{i/n})_{  i }, (\tilde{X}_{i/n})_{  i })$). Before stating our main results, let us explain briefly our approach.

We will use the Malliavin calculus developed in  \cite{BGJ} and \cite{BJ83} and follow the methodology proposed in \cite{CG18} with some modifications. This requires some regularity assumptions on the coefficients $a$ and $b$. We assume that $a$ and $b$ are real functions satisfying the following regularity conditions. In the sequel, we use the notation $||f||_{\infty}= \sup_{x \in \R} |f(x)|$ for $f$ bounded.  We make the following assumptions.

{\bf HR} : the functions $a$ and $b$ are $\mathcal{C}^3$ with bounded derivatives and $a$ is lower bounded
$$
\forall x \in \R, \quad 0 < \underline{a} \leq a(x).
$$
The L\'evy process $L$  admits the decomposition 
$$
L_t= \int_0^t \int_{\R \setminus \{0\}} z 1_{\{ |z| \leq 1\} }\tilde{\mu}(ds,dz) + \int_0^t \int_{\R \setminus \{0\}} z 1_{\{ |z| >1\} }\mu(ds,dz),
$$
with $\tilde{\mu}=\mu- \overline{\mu}$, where $\mu$ is a Poisson random measure and $\overline{\mu}(dt,dz) =dt F(dz)$ its compensator. 
We assume that  $L$ satisfies assumption {\bf A} (i)  and either (ii) or (iii).

{\bf A} :  $(L_t)_{t \geq 0}$ is a L\'evy process with triplet $(0,0,F)$ with
$$
F(dz)= \frac{g(z)}{|z|^{\alpha+1} }1_{\mathbb{R}\setminus \{0\}}(z) dz, \quad \alpha \in(0,2).
$$

(i) We assume that $g : \mathbb{R} \mapsto \mathbb{R}$ is a continuous symmetric non negative bounded function with $g(0)=c_0>0$.

(ii) We assume that $g$ is differentiable on $\{|z| >0\}$ and   $g'/g$ is bounded on $\{|z| >0\}$. 

(iii) We assume that $g$ is  supported on $\{ |z| \leq \frac{1}{2 || a'||_{\infty}} \}$ and differentiable with $g'$ bounded on $\{0< |z| \leq \frac{1}{2 || a'||_{\infty}} \}$ and that
$$
\int_{\R} \left| \frac{g'(z)}{g(z)} \right|^p g(z) dz < \infty, \quad \forall p \geq 1.
$$
In the sequel we use the notation 

{\bf A0} : {\bf A} (i) and (ii),  

{\bf A1} : {\bf A} (i) and (iii). 

Let us make some comments on these assumptions. We remark that  {\bf A0} is satisfied by a large class of processes, in particular $\alpha$-stable processes ($g=c_0$) or 
 tempered stable processes ($g(z)=c_0e^{-\lambda \vert z \vert}$, $\lambda>0$). On the other hand, assumption {\bf A1} is very restrictive. Actually, the restriction on the support of $g$ implies the non-degeneracy  assumption (Assumption (SC) p.14 in \cite{BGJ}) that can be written in our framework
$$
\forall x,z, \quad | 1+ a'(x)z | \geq \xi >0. \quad (SC)
$$
 This condition permits to apply  Theorem \ref{T:Mal-Hell} in Section \ref{S:Malliavin} (integrability of the inverse of   $U_1^{K,n,r}$). Assumption {\bf A1} is required to deal with a non constant scale function $a$ ($||a'||_{\infty} >0$). Conversely, if $a$ is constant, then the non-degeneracy assumption (SC) is satisfied and we get our results assuming the weaker assumption {\bf A0}. We also observe that these assumptions can be relaxed and that only regularity of $g$ around zero is required to obtain a rate of convergence in total variation of $\overline{X}_{1/n}$ (or $\tilde{X}_{1/n}$) to $X_{1/n}$ (see Section \ref{S:local}).
  
Since Malliavin calculus requires integrability properties for the driving process $L$, to deal with assumption {\bf A0},  we introduce a truncation function  in order to suppress the jumps larger than a constant $K$ (the truncation is useless under {\bf A1}). In a second step we will make $K$ tend to infinity. 
So we consider the truncated L\'evy process $(L^{K}_t)_{t \geq 0}$with L\'evy measure $F^K$  defined by
$$
F^K(dz)= \tau_K(z) F(dz),
$$
where $F$ is the L\'evy measure of $L$ and $\tau_K$ is a smooth truncation function such that $\tau_K$ is supported on $\{ |x| \leq K\}$ and equal to $1$ on $\{ |x| \leq K/2\}$.

We associate to $L^K$ the truncated process that solves
\begin{equation} \label{E:KEDS}
X_t^{K}= x_0 + \int_0^t b(X_s^{K}) ds + \int_0^t a(X_{s-}^{K}) d L_s^{K}, \quad t \in [0,1],
\end{equation}
and its discretization defined by $\overline{X}^K_0=x_0$ and (with $\xi$ defined in \eqref{E:EDO})
\begin{equation} \label{E:KEuler}
\overline{X}^K_t= \xi_{t-t_{i-1}}(\overline{X}^K_{t_{i-1}}) + a(\overline{X}^K_{t_{i-1}}) (L^K_t-L^K_{t_{i-1}}), \quad  t \in [t_{i-1},t_i], \quad 1 \leq i \leq n.
\end{equation}
Thanks to the truncation $\tau_K$, $\E |L^K_t|^p< \infty$, for any $p \geq1$, we can apply the Malliavin calculus on Poisson space introduced in \cite{BGJ}. 

Now under {\bf HR} and {\bf A0} or {\bf A1}, the random variables $X^K_t$ and $\overline{X}^K_t$ admit a density for $t>0$ (see \cite{BJ83}). Note that under {\bf A1}, $X=X^K$ and $X=\overline{X}^K$ for $K$ large enough. Let $p_{1/n}^K$, respectively $\overline{p}^K_{1/n}$, be the transition density of the Markov chain $(X^K_{i/n})_{i \geq 0}$, respectively $(\overline{X}^K_{i/n})_{i \geq 0}$.
From Proposition \ref{P:Tenso}, we have 
\begin{equation} \label{E:TVK}
d_{TV} ((X^K_{\frac{i}{n}}), (\overline{X}^K_{\frac{i}{n}})) \leq \left( 
\frac{1}{2}\sum_{i=1}^n \left( \E H^2_{X^K_{\frac{i-1}{n}}}(p^K_{1/n},\overline{p}^K_{1/n}) +\E H^2_{\overline{X}^K_{\frac{i-1}{n}}}(p^K_{1/n},\overline{p}^K_{1/n}) \right) \right)^{1/2}.
\end{equation}
Consequently to bound the total variation distance between $(X^K_{\frac{i}{n}})_{0 \leq  i \leq n}$ and $(\overline{X}^K_{\frac{i}{n}})_{0 \leq  i \leq n}$ it is sufficient to control $H_x(p_{1/n}^K, \overline{p}^K_{1/n})$ in terms of $n$, $K$ and $x$. Bounds for $H_x(p_{1/n}^K, \overline{p}^K_{1/n})$ are presented in the next section. They are obtained by connecting $H_x(p_{1/n}^K, \overline{p}^K_{1/n})$ to the $L^2$-norm of a Malliavin weight. This technical part of the paper is postponed to Section \ref{S:Malliavin}.

Of course, the methodology is exactly the same if we replace the scheme $\overline{X}$ by the Euler scheme $\tilde{X}$. In that case we consider the truncated Euler scheme defined by $\tilde{X}^K_0=x_0$ and for $t \in [t_{i-1},t_i]$, $1 \leq i \leq n$,
\begin{equation} \label{E:KTEuler}
\tilde{X}^K_t= \tilde{X}^K_{t_{i-1}}+b(\tilde{X}^K_{t_{i-1}})(t-t_{i-1}) + a(\tilde{X}^K_{t_{i-1}}) (L^K_t-L^K_{t_{i-1}}).
\end{equation}
We denote by $\tilde{p}^K_{1/n}$ the transition density of the Markov chain $(\tilde{X}^K_{i/n})_{i \geq 0}$.

Throughout the paper, $C(a,b, \alpha)$ (or $C(a,b)$, $C(b)$, $C(\alpha)$) denotes a constant, whose value may change from line to line, independent of $n$, $K$  but depending on the functions $a$, $b$ and the index $\alpha$. We write simply $C$ if $C(a,b, \alpha)$ does not depend on $a$, $b$, $\alpha$. The constant may depend on other fixed  parameters such as the parameter $p$ in H\"{o}lder's inequality and we omit in general this dependence except if some optimal choices are required, such as $p=1+ \varepsilon$ for $\varepsilon$ arbitrarily  small, in that case we use the notation  $C_{\varepsilon}(a,b, \alpha)$. 

\section{Estimates for the local Hellinger distance} \label{S:RateH}

We state in this section our main results concerning the rate of convergence in approximating $X^K_{1/n}$ solution of \eqref{E:KEDS} starting from $x$, by $\overline{X}^K_{1/n}$ or $\tilde{X}^K_{1/n}$ that solve respectively \eqref{E:KEuler} or \eqref{E:KTEuler} with initial value $x$. In what follows, the constant $C(a,b, \alpha)$ does not depend on $x$.

Before stating our results, we precise the assumptions on the auxiliary truncation $\tau_K$. Let $\tau$ be a symmetric $\mathcal{C}^1$ function such that $0 \leq \tau(x) \leq 1$, $\tau(x)=1$ if $|x| \leq 1/2$ and $\tau(x)=0$ if $|x| \geq 1$. We assume moreover that 
\begin{equation} \label{E:tau}
\forall p \geq 1 , \quad \int \left|  \frac{\tau'(z)}{\tau(z)} \right|^p \tau(z) dz < \infty.
\end{equation}
For $K \geq \kappa_0 >0$, we define $\tau_K$ by $\tau_K(x)=\tau(x/K)$. 

We first assume  that $a$ is constant.
In that case, our methodology does not require additional non-degeneracy assumptions on the L\'evy measure and we assume {\bf A0}.
We present in the next theorem the bounds obtained for the schemes defined by 
 \eqref{E:KEuler} and \eqref{E:KTEuler}. In general, the bound depends on the truncation $K$ but if $g$ satisfies the additional integrability assumption 
$\int |z| g(z) dz < \infty$ then the bound is independent of $K$. We observe also that the bound is slightly better in the stable case.

\begin{thm} \label{T:BHellHR0}
We assume {\bf A0} and {\bf HR} with $a$ constant. 

(i) For the scheme  \eqref{E:KEuler}, for any $\alpha \in (0,2)$ we have
$$
\sup_x H^2_x(p_{1/n}^K, \overline{p}^K_{1/n}) \leq 
 \begin{cases}
\frac{C(a,b, \alpha)}{n^2} (1 + \frac{K^{2- \alpha}}{n} ), \\
 \frac{C(a,b, \alpha)}{n^2},  \mbox{ if  } \int |z| g(z) dz < \infty, \\
\frac{C(a,b, \alpha)}{n^2} (1 + \frac{K^{2- \alpha}}{n^3} ),  \mbox{ in the stable case } g=c_0.
\end{cases}
$$
(ii) For the Euler scheme \eqref{E:KTEuler}, we have for $\alpha >1/2$
$$
 H^2_x(p_{1/n}^K, \tilde{p}^K_{1/n}) \leq 
 \begin{cases}
 \frac{C(a,b, \alpha)}{n^2} (1 + \frac{K^{2- \alpha}}{n} + |b(x)|^2 \frac{n^{2/ \alpha}}{n^2}), \\
 \frac{C(a,b, \alpha)}{n^2}(1+ |b(x)|^2 \frac{n^{2/ \alpha}}{n^2}), \mbox{ if  } \int |z| g(z) dz < \infty.
\end{cases}
$$
In (i) and (ii), $C(a,b,\alpha)$ has exponential growth in $||b'||_{\infty}$ and polynomial growth in $||b''||_{\infty}$, $1/a$, $a$, $1/ \alpha$ and $1/(\alpha-2)$.
\end{thm}




In the general case ($a$ non constant), we need strong restrictions on the support of the L\'evy measure $F$ and assume 
 {\bf A1}. So we have $X^K=X$ and $\overline{X}^K= \overline{X}$ for $K$ large enough and we omit the dependence on $K$.

\begin{thm} \label{T:BHellG}
We assume {\bf A1} and {\bf HR} with $||a' ||_{\infty}>0$, then we have

(i)
$$
 H^2_x(p_{1/n}, \overline{p}_{1/n}) \leq 
 \begin{cases}
C(a,b, \alpha) (1+ |x|^2)\frac{1}{n^{2/ \alpha}} , \quad \mbox{if} \quad \alpha >1,\\
C_{\varepsilon}(a,b, \alpha) (1+ |x|^2)\frac{1}{n^{2- \varepsilon}}, \quad \mbox{if} \quad \alpha \leq 1, \; \forall \varepsilon >0,
 \end{cases}
$$
(ii) For the Euler scheme \eqref{E:TrueEuler}, we obtain for $\alpha >1/2$
$$
 H^2_x(p_{1/n}, \tilde{p}_{1/n}) \leq 
 \begin{cases}
C(a,b, \alpha) (1+ |x|^2)\frac{1}{n^{2/ \alpha}} , \quad \mbox{if} \quad \alpha >1,\\
C_{\varepsilon}(a,b, \alpha) (1+ |x|^2)\frac{1}{n^{2- \varepsilon}}, \quad \mbox{if} \quad \alpha =1, \; \forall \varepsilon >0, \\
C(a,b, \alpha) (1+ |x|^2)\frac{1}{n^{4-2/ \alpha}} , \quad \mbox{if} \quad 1/2<\alpha <1.
 \end{cases}
$$
In (i) and (ii), $C(a,b,\alpha)$ (or $C_{\varepsilon}(a,b,\alpha)$) has exponential growth in $||b'||_{\infty}$ and polynomial growth in $||b''||_{\infty}$, $||a'||_{\infty}$, $||a''||_{\infty}$, $1/||a'||_{\infty}$, $b(0)$, $a(0)$, $1/\underline{a}$, $1/ \alpha$ and $1/(\alpha-2)$.

\end{thm}

\begin{rem}
In the Brownian case ($\alpha=2$), we obtain the rate of convergence $1/n$ for the square of the Hellinger distance between $X_{1/n}$ and its Euler approximation $\tilde{X}_{1/n}$. This  sharp rate (see Remark \ref{R:brownien}) does not permit to obtain a path control of the total variation distance between the stochastic equation and  the Euler scheme. This is why we focus in this paper on pure jump processes. To obtain pathwise convergence in the Brownian case, one has to consider a discretization scheme with finer step as in Konakov and al. \cite{Woerner14}.
\end{rem}

The proof of these theorems is given in Sections \ref{S:ProofC} and \ref{S:ProofG}.  

\section{Total variation distance : rate of convergence and examples} \label{S:TV}

\subsection{Pathwise total variation} \label{S:TVpath}

The local behavior of the Hellinger distance established in Section \ref{S:RateH} permits to obtain some pathwise rates of convergence
 in  total variation.
As in the previous section, we distinguish between the cases $a$ constant (where the rate of convergence is better) or $a$ non constant and we study rate of convergence for  the total variation distance between $(X_{i/n})_{0 \leq  i \leq n}$ and $(\overline{X}_{i/n})_{0 \leq  i \leq n}$ (respectively $(\tilde{X}_{i/n})_{0 \leq  i \leq n}$) defined by \eqref{E:EDS} and \eqref{E:Euler} (respectively \eqref{E:TrueEuler}).

\begin{thm} \label{T:BTVHR0}
We assume {\bf A0} and {\bf HR} with $a$ constant.

(i) For the scheme \eqref{E:Euler}, we have 
\begin{equation*} \label{E:BTVHR0}
d_{TV} ((X_{\frac{i}{n}})_{0 \leq  i \leq n}, (\overline{X}_{\frac{i}{n}})_{0 \leq  i \leq n}) \leq 
\begin{cases}
C(a,b, \alpha)
\max( \frac{1}{\sqrt{n}}, \frac{1}{n^{2 \alpha/(\alpha+2)}}), \\
\frac{C(a,b, \alpha)}{ \sqrt{n}}, \mbox{ if } \int_{\R} |z| g(z) dz < \infty, \\
C(a,b, \alpha)
\max( \frac{1}{\sqrt{n}}, \frac{1}{n^{4 \alpha/(\alpha+2)}}), \mbox{ in the stable case } g=c_0,
\end{cases}
\end{equation*}
where $C(a,b,\alpha)$ has exponential growth in $||b'||_{\infty}$ and polynomial growth in $||b''||_{\infty}$, $1/a$, $a$, $1/ \alpha$ and $1/(\alpha-2)$.

(ii) For the Euler scheme \eqref{E:TrueEuler}, we have for $\alpha > 2/3$
\begin{equation*} \label{E:BTVHR0Euler}
d_{TV} ((X_{\frac{i}{n}})_{0 \leq  i \leq n}, (\tilde{X}_{\frac{i}{n}})_{0 \leq  i \leq n}) \leq C(a,b,\alpha)
\max( \frac{1}{\sqrt{n}}, \frac{1}{n^{\frac{3\alpha -2}{\alpha+2}}}).
\end{equation*}
Moreover with the additional assumption on $g$,
$
\int_{\R} |z| g(z) dz < \infty, 
$
then  
$$
d_{TV} ((X_{\frac{i}{n}})_{0 \leq  i \leq n}, (\tilde{X}_{\frac{i}{n}})_{0 \leq  i \leq n}) \leq 
\begin{cases}
C(a,b,\alpha) \frac{1}{\sqrt{n}}, \quad \mbox{if} \quad \alpha \geq 1, \\
C(a,b,\alpha)\frac{1}{n^{\frac{3}{2} -\frac{1}{\alpha} }}, \quad \mbox{if} \quad \frac{2}{3} < \alpha < 1.
\end{cases}
$$

\end{thm}

\begin{rem}
(i) We  observe that without integrability assumptions on $g$, the rate of convergence vanishes if $\alpha$ goes to zero. Moreover we have $\max( \frac{1}{\sqrt{n}}, \frac{1}{n^{2 \alpha/(\alpha+2)}})= \frac{1}{\sqrt{n}}$ if $\alpha\geq 2/3$.  In the stable case, the rate $ \frac{1}{\sqrt{n}}$ is obtained if $\alpha \geq 2/7$. 

(ii) For the Euler scheme, we have no rate at all if $\alpha \leq 2/3$.
\end{rem}

\begin{rem}
We can apply our methodology if the L\'evy process $L$ is a Brownian Motion. In that case the Malliavin calculus is more standard and we compute easily the Malliavin weight of Section \ref{S:Malliavin}. Assuming {\bf HR} and $a$ constant, we obtain the rate of convergence  $1/ \sqrt{n}$ in total variation distance between $(X_{\frac{i}{n}})_{0 \leq  i \leq n}$ and  $(\tilde{X}_{\frac{i}{n}})_{0 \leq  i \leq n}$.
\end{rem}

\begin{proof}[Proof of Theorem \ref{T:BTVHR0}]

{\bf (i)} We first establish a relationship between the total variation distance 
$d_{TV} ((X_{i/n})_{  i }, (\overline{X}_{i/n})_{  i })$ and 
$d_{TV} ((X^K_{i/n})_{  i }, (\overline{X}^K_{i/n})_{  i })$. On the same probability space $(\Omega, \mathcal{F}, (\mathcal{F}_t), \mathbb{P})$ we consider the L\'evy process $(L_t)_{t \geq 0}$ with L\'evy measure $F$ and the truncated L\'evy process $(L^{K}_t)_{t \geq 0}$ with L\'evy measure $F^K$  defined by
$$
F^K(dz)= \tau_K(z) F(dz).
$$
We recall (see Section 4.1  in \cite{CG18}) that this can be done  by setting
$L_t= \int_0^t \int_{\mathbb{R} }z 1_{\{|z| \leq 1 \}} \tilde{\mu}(ds,dz) + \int_0^t \int_{\mathbb{R} }z 1_{\{|z| > 1 \}} \mu(ds,dz) $, respectively $L_t^{K}= \int_0^t \int_{\mathbb{R}} z \tilde{\mu}^{K}(ds,dz)$, where $\tilde{\mu}$, respectively $\tilde{\mu}^{K}$, are the compensated Poisson random measures associated respectively to
$$
\mu(A) = \int_{[0,1]} \int_{\mathbb{R}} \int_{[0,1]} 1_A(t,z) \mu^*( dt, dz, du), \quad A \subset [0,1] \times \mathbb{R}
$$
$$
\mu^{K}(A) = \int_{[0,1]} \int_{\mathbb{R}} \int_{[0,1]} 1_A(t,z) 1_{ \{u \leq \tau_K(z) \}}\mu^*( dt, dz, du), \quad A \subset [0,1] \times \mathbb{R},
$$
for $\mu^*$ a Poisson random measure on $[0,1]\times \mathbb{R} \times [0,1]$ with compensator $\overline{\mu}^* (dt, dz, du) =dt F(dz) du$.
By construction,  the measures $\mu$  and $\mu^{K}$  coincide on the event
\begin{equation} \label{E:Omega_n}
\Omega_K= \{ \omega \in \Omega ;\mu^*(  [0,1] \times \{ z \in \mathbb{R} ; \abs{z} \geq K/2 \} \times [0,1] )=0 \}.
\end{equation}
Since $\mu^*(  [0,1] \times \{ z \in \mathbb{R} ; \abs{z} \geq K/2 \} \times [0,1] )$ has a Poisson distribution with parameter 
$$
\lambda_K=\int_{ \abs{z} \geq K/2} g(z)/\abs{z}^{\alpha+1} dz \leq C/(\alpha K^{\alpha}),
$$
we deduce that 
\begin{equation} \label{E:OmC}
\mathbb{P}( \Omega_K^c) \leq C(\alpha)/ K^{\alpha}.
\end{equation}
We observe that  $(X_t,\overline{X}_t, L_t)_{t \in [0,1]}= (X_t^{K}, \overline{X}^K_t,L_t^{K})_{t \in [0,1]}$ on $\Omega_K$ and so we deduce
\begin{equation}
d_{TV} ((X_{\frac{i}{n}})_{0 \leq  i \leq n}, (\overline{X}_{\frac{i}{n}})_{0 \leq  i \leq n}) \leq d_{TV} ((X^K_{\frac{i}{n}})_{0 \leq  i \leq n}, (\overline{X}^K_{\frac{i}{n}})_{0 \leq  i \leq n}) +C(\alpha)/ K^{\alpha}. \label{E:TV}
\end{equation}

\underline{General bound}. Combining \eqref{E:TV}, \eqref{E:TVK}  with Theorem \ref{T:BHellHR0} (i) we have
\begin{eqnarray*}
d_{TV} ((X_{\frac{i}{n}})_{0 \leq  i \leq n}, (\overline{X}_{\frac{i}{n}})_{0 \leq  i \leq n}) & \leq &\frac{C(a,b, \alpha)}{\sqrt{n}} (1+ \frac{K^{2- \alpha}}{n} )^{1/2} + \frac{C(\alpha)}{K^{\alpha}} \\
& \leq & C(a,b,\alpha) (\frac{1}{\sqrt{n}} + \frac{K^{1- \alpha/2}}{n} + \frac{1}{K^{\alpha}}).
\end{eqnarray*}
Choosing $K=n^{2/( \alpha+2)}$, we deduce 
$$
 \frac{K^{1- \alpha/2}}{n}=\frac{1}{n^{2 \alpha/(\alpha+2)}}= \frac{1}{K^{\alpha}},
$$
this gives the first part of the result.

\underline{With the integrability assumption on $g$}. We have
$$
d_{TV} ((X_{\frac{i}{n}})_{0 \leq  i \leq n}, (\overline{X}_{\frac{i}{n}})_{0 \leq  i \leq n})  \leq \frac{C(a,b,\alpha)}{\sqrt{n}} + \frac{C(\alpha)}{K^{\alpha}}, 
$$
and we conclude choosing $K=n^{1/(2\alpha)}$.

\underline{In the stable case}. We have
$$
d_{TV} ((X_{\frac{i}{n}})_{0 \leq  i \leq n}, (\overline{X}_{\frac{i}{n}})_{0 \leq  i \leq n})  \leq C(a,b,\alpha)( \frac{1}{\sqrt{n}} +\frac{K^{1- \alpha/2}}{n^2}+ \frac{C(\alpha)}{K^{\alpha}}).
$$
We conclude with $K=n^{4/( \alpha+2)}$.


{\bf (ii)} From \eqref{E:TVK} and Theorem \ref{T:BHellHR0} (ii) we have
\begin{eqnarray*}
d_{TV} ((X^K_{\frac{i}{n}})_{0 \leq  i \leq n}, (\tilde{X}^K_{\frac{i}{n}})_{0 \leq  i \leq n}) & \leq  & \frac{C(a,b,\alpha)}{\sqrt{n}}\left( 1+ \frac{K^{2-\alpha}}{n} \right. \\
 &  & \left. + [\sup_{t \in [0,1] }\E |X^K_t|^2 + \sup_{t \in [0,1] }\E |\tilde{X}^K_t|^2 ]\frac{n^{2/ \alpha}}{n^2} \right)^{1/2}.
\end{eqnarray*}
Standard computations give
$$
\sup_{t \in [0,1] } \E | X^K_t|^2 \leq C(a,b,\alpha) K^{2- \alpha}, \quad \quad \sup_{t \in [0,1] } \E | \tilde{X}^K_t|^2 \leq C(a,b,\alpha) K^{2- \alpha}.
$$
So we obtain
$$
d_{TV} ((X^K_{\frac{i}{n}})_{0 \leq  i \leq n}, (\tilde{X}^K_{\frac{i}{n}})_{0 \leq  i \leq n})  \leq \frac{C(a,b,\alpha)}{\sqrt{n}}(1+ K^{1- \alpha/2}\frac{n^{1/ \alpha}}{n}).
$$
Now proceeding as in the beginning of the proof of Theorem \ref{T:BTVHR0}, we see that \eqref{E:TV} holds, replacing $\overline{X}$ by $\tilde{X}$, and we deduce
$$
d_{TV} ((X_{\frac{i}{n}})_{0 \leq  i \leq n}, (\tilde{X}_{\frac{i}{n}})_{0 \leq  i \leq n})  \leq \frac{C(a,b,\alpha)}{\sqrt{n}}(1+ K^{1- \alpha/2}\frac{n^{1/ \alpha}}{n}) + \frac{C(\alpha)}{K^{\alpha}}.
$$
Choosing $K=n^{(3\alpha-2)/(\alpha (2 + \alpha))}$ gives the first result.

\underline{With the integrability assumption on $g$}. The $L^2$-norm of $(X^K_t)$ and $(\tilde{X}^K_t)$ does not depend on $K$ and  we have
$$
\sup_{t \in [0,1] } \E | X^K_t|^2 \leq C(a,b,\alpha) , \quad \quad \sup_{t \in [0,1] } \E | \tilde{X}^K_t|^2 \leq  C(a,b,\alpha).
$$
So it yields
$$
d_{TV} ((X_{\frac{i}{n}})_{0 \leq  i \leq n}, (\tilde{X}_{\frac{i}{n}})_{0 \leq  i \leq n})  \leq \frac{C(a,b,\alpha)}{\sqrt{n}}(1+ \frac{n^{1/ \alpha}}{n}) + \frac{C(\alpha)}{K^{\alpha}}.
$$
With $K=n^{1/(2 \alpha)}$ we deduce
$$
d_{TV} ((X_{\frac{i}{n}})_{0 \leq  i \leq n}, (\tilde{X}_{\frac{i}{n}})_{0 \leq  i \leq n})  \leq C(a,b, \alpha) \max( \frac{1}{\sqrt{n}}, \frac{1}{n^{(3 \alpha-2)/(2 \alpha)}}).
$$

\end{proof}





We now study  the convergence rate in total variation distance for a general scale coefficient $a$, assuming {\bf A1}.  We observe that in the Brownian case $\alpha=2$, we do not have convergence.

\begin{thm} \label{T:BTVHR}
We assume {\bf A1} and  {\bf HR} with  $||a' ||_{\infty} >0$.

(i) Then we have 
\begin{equation*} \label{E:BTVHR}
d_{TV} ((X_{\frac{i}{n}})_{0 \leq  i \leq n}, (\overline{X}_{\frac{i}{n}})_{0 \leq  i \leq n}) \leq 
\begin{cases}
C(a,b, \alpha) \frac{1}{n^{1/\alpha -1/2}}, \quad \mbox{if} \quad \alpha >1, \\ 
C_{\varepsilon}(a,b, \alpha) \frac{1}{n^{1/2 - \varepsilon}} \quad \mbox{if} \quad \alpha \leq 1, \; \forall \varepsilon >0. 
\end{cases}
\end{equation*}
where $C(a,b,\alpha)$ (or $C_{\varepsilon}(a,b, \alpha)$) has exponential growth in $||b'||_{\infty}$ and polynomial growth in $||b''||_{\infty}$, $||a'||_{\infty}$, $||a''||_{\infty}$, $1/||a'||_{\infty}$, $b(0)$, $a(0)$, $1/\underline{a}$, $1/ \alpha$ and $1/(\alpha-2)$.

(ii) For the Euler scheme \eqref{E:TrueEuler}, we obtain  if $\alpha > 2/3$
$$
d_{TV} ((X_{\frac{i}{n}})_{0 \leq  i \leq n}, (\tilde{X}_{\frac{i}{n}})_{0 \leq  i \leq n}) \leq 
\begin{cases}
C(a,b, \alpha) \frac{1}{n^{1/\alpha -1/2}}, \quad \mbox{if} \quad \alpha >1, \\ 
C_{\varepsilon}(a,b, \alpha) \frac{1}{n^{1/2 - \varepsilon}} \quad \mbox{if} \quad \alpha =1, \; \forall \varepsilon >0, \\
C(a,b, \alpha) \frac{1}{n^{3/2 - 1/ \alpha}} \quad \mbox{if} \quad 2/3 <\alpha <1.
\end{cases}
$$
\end{thm}
\begin{proof}
Under {\bf A1}, $g$ is a truncation function and the result is an immediate consequence of \eqref{E:TVK} and Theorem \ref{T:BHellG} observing that for any $p\geq 1$
$$
\sup_{t \in [0,1] } \E | X_t|^p \leq C(a,b,\alpha), \quad  \sup_{t \in [0,1] } \E | \overline{X}_t|^p \leq C(a,b,\alpha), \quad 
\sup_{t \in [0,1] } \E | \tilde{X}_t|^p \leq C(a,b,\alpha).
$$
\end{proof}

\begin{rem}
The result of Theorems \ref{T:BTVHR0} and \ref{T:BTVHR} has interesting consequences in statistics. Assume that $b$ and $a$ depend on  unknown parameters $\theta$ and $\sigma$ and that we are interested in estimating the three parameters $\beta=(\theta, \sigma, \alpha)$. Let $\mathcal{E}^{n}$ be the experiment based on the observations $(X^{\beta}_{\frac{i}{n}})_{0 \leq  i \leq n}$ given by \eqref{E:EDS}  and let $\overline{\mathcal{E}}^{n}$   be the experiment based on the observations $(\overline{X}^{\beta}_{\frac{i}{n}})_{0 \leq  i \leq n}$ given by \eqref{E:Euler}. With additional assumptions on the coefficients $a$ and $b$, we can prove that the total variation distance between the two experiments goes to zero, uniformly with respect to $\beta$, and consequently statistical inference in experiment $\mathcal{E}^n$ inherits the same asymptotic properties as in experiment $\overline{\mathcal{E}}^n$. Efficiency in $\mathcal{E}^n$ is still an open problem for a general scale coefficient $a$ (assuming $a$ constant, the LAMN property for $(\theta,a)$ has been established in \cite{CGN2} assuming additionally that $(L_t)$ is a truncated stable process). The main difficulty comes from the fact that the likelihood function is not explicit. But using the asymptotic equivalence of  $\mathcal{E}^n$ and  $\overline{\mathcal{E}}^n$, it is sufficient to study asymptotic efficiency in the simplest experiment $\overline{\mathcal{E}}^n$ where the likelihood function has an explicit expression in term of the density of the driving L\'evy process. 
\end{rem}

\subsection{Local total variation} \label{S:local}
The local estimates in Hellinger distance give bounds for the local total variation distance and permit to extend the results obtained in \cite{CG18} where the Euler scheme and the case $a$ constant were not considered (only (i) in Proposition \ref{P:BTVlocG} below is considered in \cite{CG18} and we slightly improve  here the bound for $\alpha >1$). So in this section we precise the
bounds for the total variation distance between $X_{1/n}$ and  $\overline{X}_{1/n}$, or $\tilde{X}_{1/n}$, that we deduce from the results of Section \ref{S:RateH}. Since we  consider approximation in small time, we do not need to make the truncation $K$ tend to infinity, consequently we can relax the assumptions on the L\'evy measure and only assume regularity around zero. We now assume that $L$ satisfies assumption {\bf AL} below.

{\bf AL} :    $(L_t)_{t \geq 0}$ is a L\'evy process with triplet $(0,0,F)$ with
$$
F(dz)= \frac{g(z)}{|z|^{\alpha+1} }1_{\{0< |z| < \eta \}} dz+ F_1(dz), \quad \alpha \in(0,2), \quad \eta >0,
$$
where $F_1$ is a symmetric finite measure supported on $\{ |z| \geq \eta\}$ and
$g$  a continuous symmetric non negative bounded function on $\{ |z| < \eta\}$, with $g(0)=c_0>0$. We also assume that $g$ is continuously differentiable on $\{0<|z| < \eta\}$ with   $g'/g$  bounded on $\{0<|z| < \eta\}$. 

We summarize our results in the next propositions.

\begin{prop} \label{P:BTVloc}
We assume {\bf AL} and {\bf HR} with $a$ constant. 

(i) For the scheme  \eqref{E:Euler}, for any $\alpha \in (0,2)$ we have
$$
\sup_x d_{TV}(X_{1/n}, \overline{X}_{1/n}) \leq 
\frac{C(a,b, \alpha)}{n}. 
$$
(ii) For the Euler scheme \eqref{E:TrueEuler}, we have for $\alpha >1/2$
$$
 d_{TV}(X_{1/n}, \tilde{X}_{1/n})\leq 
 \begin{cases}
 \frac{C(a,b, \alpha)}{n} (1 + |b(x)|), \mbox{ if  } \alpha >1, \\
 \frac{C(a,b, \alpha)}{n^{2-1/ \alpha}}(1+ |b(x)|), \mbox{ if  } 1/2<\alpha \leq 1.
\end{cases}
$$
\end{prop}

\begin{prop} \label{P:BTVlocG}
We assume {\bf AL} and {\bf HR} with $||a' ||_{\infty}>0$.

(i) For the scheme  \eqref{E:Euler}, we have
$$
d_{TV}(X_{1/n}, \overline{X}_{1/n}) \leq 
 \begin{cases}
C(a,b, \alpha) (1+ |x|)\frac{1}{n^{1/ \alpha}} , \quad \mbox{if} \quad \alpha >1,\\
C_{\varepsilon}(a,b, \alpha) (1+ |x|)\frac{1}{n^{1- \varepsilon}}, \quad \mbox{if} \quad \alpha \leq 1, \; \forall \varepsilon >0.
 \end{cases}
$$
(ii) For the Euler scheme \eqref{E:TrueEuler}, we obtain for $\alpha >1/2$
$$
 d_{TV}(X_{1/n}, \tilde{X}_{1/n}) \leq 
 \begin{cases}
C(a,b, \alpha) (1+ |x|)\frac{1}{n^{1/ \alpha}} , \quad \mbox{if} \quad \alpha >1,\\
C_{\varepsilon}(a,b, \alpha) (1+ |x|)\frac{1}{n^{1- \varepsilon}}, \quad \mbox{if} \quad \alpha =1, \; \forall \varepsilon >0, \\
C(a,b, \alpha) (1+ |x|)\frac{1}{n^{2-1/ \alpha}} , \quad \mbox{if} \quad 1/2<\alpha <1.
 \end{cases}
$$
\end{prop}

\begin{proof}[Proof of Propositions \ref{P:BTVloc} and \ref{P:BTVlocG}]
We consider the truncation $\tau_K$ defined at the beginning of Section \ref{S:RateH}.  If $a$ is constant, we fix $0<K < \eta$, consequently Theorem \ref{T:BHellHR0} holds. In the case   $||a' ||_{\infty}>0$, we fix  $0<K < \min( \eta, \frac{1}{2 || a'||_{\infty}})$, then  {\bf A1} is satisfied for $g \tau_K$ and we can apply Theorem \ref{T:BHellG}.

 Proceeding as in the proof of Theorem \ref{T:BTVHR0} (i), we can define the processes on the same probability space such that 
 $(X_{1/n},\overline{X}_{1/n},\tilde{X}_{1/n})= (X_{1/n}^{K}, \overline{X}^K_{1/n}, \tilde{X}^K_{1/n})$ on an event $\Omega_{K,n}$ with
 \begin{equation*}
\mathbb{P}( \Omega_{K,n}^c) \leq \frac{C(\alpha)}{n }.
\end{equation*}
The constant depends on $K$ but since $K$ is fixed we omit it.
The result follows then immediately from Theorems \ref{T:BHellHR0} and \ref{T:BHellG}.

\end{proof}
\subsection{Examples} \label{S:ex}
To end this section, we discuss the optimality of the previous upper bounds by establishing lower bounds for the local total variation distance for specific stochastic equations. We consider an Ornstein-Uhlenbeck process driven by a stable L\'evy process and the stochastic exponential. 

\quad

\noindent
{\it Stable Ornstein-Uhlenbeck process.}
We assume that $(X_t)_{t \geq 0}$ solves the equation
\begin{equation} \label{E:OU}
X_t=x - \int_0^t X_s ds + S_t^{\alpha}, \quad x \neq 0
\end{equation}
where $(S_t^{\alpha})_{t \geq 0}$ is a stable process with characteristic function $\E(e^{iu S_1^{\alpha}})=e^{-|u|^{\alpha}}$, $\alpha \in (0,2)$. 
The next result shows that the rates of Proposition \ref{P:BTVloc} are reached.
\begin{prop}
For the stable Ornstein-Uhlenbeck process \eqref{E:OU}, we have for $n$ large enough
$$
d_{TV}(X_{1/n},\overline{X}_{1/n}) \geq \frac{C(\alpha)}{n}
$$
and for the Euler scheme
$$
d_{TV}(X_{1/n},\tilde{X}_{1/n}) \geq \begin{cases}
\frac{C(\alpha)}{n} \; \mbox{ if } \alpha > 1\\
\frac{C(\alpha)}{n^{2-1/ \alpha}} \; \mbox{ if } 1/2 < \alpha \leq 1
\end{cases}
$$
where $C(\alpha) >0$ and depends on $x$ for the Euler scheme.
\end{prop}

\begin{proof} For this process, we can check (using the scaling property of the stable distribution) that $X_{1/n}$, $\overline{X}_{1/n}$ and  $\tilde{X}_{1/n}$ have the following distributions :
$$
X_{\frac{1}{n}}  \stackrel{\mathcal{L}}{=} xe^{-\frac{1}{n}} + e^{-\frac{1}{n} } \left(\int_0^{\frac{1}{n}} e^{\alpha u} du\right)^{1/ \alpha} S_1^{\alpha} = xe^{-\frac{1}{n}} + \left( \frac{1-e^{-\alpha/n}}{\alpha}\right)^{1/ \alpha} S_1^{\alpha},
$$
$$
\overline{X}_{\frac{1}{n}}   \stackrel{\mathcal{L}}{=}  xe^{-\frac{1}{n}} + \frac{1}{n^{1/ \alpha}} S_1^{\alpha}, \quad \quad
\tilde{X}_{\frac{1}{n}}  \stackrel{\mathcal{L}}{=} x(1- \frac{1}{n} ) + \frac{1}{n^{1/ \alpha}} S_1^{\alpha}.
$$
We denote by $\varphi_{\alpha}$ the density of the stable variable $S_1^{\alpha}$ and
we set 
$\sigma_{0,n}=\frac{1}{n^{1/ \alpha}}$,  $\sigma_{n}=\left( \frac{1-e^{-\alpha/n}}{\alpha}\right)^{1/ \alpha}$.
We check easily that
\begin{equation} \label{E:DLOU}
\frac{\sigma_{0,n} }{ \sigma_{n}}=1+ \frac{1}{2n} + o(\frac{1}{n} ).
\end{equation} 
With this notation, we have
$$
p_{1/n}(x,y)=\frac{1}{\sigma_n} \varphi_{\alpha}\left( \frac{y- xe^{-\frac{1}{n}}}{\sigma_n}\right), \quad 
\overline{p}_{1/n}(x,y)=\frac{1}{\sigma_{0,n}}\varphi_{\alpha}\left( \frac{y- xe^{-\frac{1}{n}}}{\sigma_{0,n}}\right), 
$$
$$
\tilde{p}_{1/n}(x,y)=\frac{1}{\sigma_{0,n}}\varphi_{\alpha}\left( \frac{y- x(1-\frac{1}{n})}{\sigma_{0,n}}\right)
$$
Consequently, we obtain
\begin{eqnarray*}
d_{TV}(X_{1/n},\overline{X}_{1/n})&=  \frac{1}{2} \int_{\R} | \frac{\varphi_{\alpha}\left( \frac{y- xe^{-\frac{1}{n}}}{\sigma_n}\right)}{\sigma_{n} }
- \frac{\varphi_{\alpha}\left( \frac{y- xe^{-\frac{1}{n}}}{\sigma_{0,n}}\right)}{\sigma_{0,n}} | dy \\
&=  \frac{1}{2} \int_{\R} | \frac{\sigma_{0,n} }{ \sigma_{n}}\varphi_{\alpha}\left( \frac{\sigma_{0,n}}{\sigma_n} y\right)
- \varphi_{\alpha}\left( y\right) | dy  \\
& \geq   \frac{1}{2} \int_0^1  | \frac{\sigma_{0,n} }{ \sigma_{n}}\varphi_{\alpha}\left( \frac{\sigma_{0,n}}{\sigma_n} y\right)
- \varphi_{\alpha}\left( y\right) | dy.
\end{eqnarray*}
Since $\varphi_{\alpha}$ is continuously differentiable, we have the expansion
$$
\varphi_{\alpha}\left( \frac{\sigma_{0,n}}{\sigma_n} y\right)=\varphi_{\alpha}\left( y\right)+y(\frac{\sigma_{0,n}}{\sigma_n}  -1) \varphi_{\alpha}^{\prime}(c_{y,n}), \quad c_{y,n} \in (y, \frac{\sigma_{0,n}}{\sigma_n} y),
$$
and we deduce
\begin{eqnarray*}
d_{TV}(X_{1/n},\overline{X}_{1/n})& \geq   \frac{1}{2} | \frac{\sigma_{0,n}}{\sigma_n}  -1 | \int_0^1  |\varphi_{\alpha}\left( y\right)+y \frac{\sigma_{0,n} }{ \sigma_{n}}\varphi_{\alpha}^{\prime}\left( c_{y,n}\right) | dy.
\end{eqnarray*}
But by dominated convergence 
$$
\int_0^1  |\varphi_{\alpha}\left( y\right)+y \frac{\sigma_{0,n} }{ \sigma_{n}}\varphi_{\alpha}^{\prime}\left( c_{y,n}\right) | dy \xrightarrow{ n \rightarrow  \infty}  \int_0^1  |\varphi_{\alpha}\left( y\right)+y \varphi_{\alpha}^{\prime}\left(y\right) | dy >0,
$$
and we conclude using \eqref{E:DLOU} that $d_{TV}(X_{1/n},\overline{X}_{1/n}) \geq \frac{C(\alpha)}{n}$.

For the Euler scheme, we have similarily
\begin{eqnarray*}
d_{TV}(X_{1/n},\tilde{X}_{1/n})&
& \geq   \frac{1}{2} \int_0^1  | \frac{\sigma_{0,n} }{ \sigma_{n}}\varphi_{\alpha}\left( \frac{\sigma_{0,n}}{\sigma_n} y+ x\frac{1-\frac{1}{n}-e^{-\frac{1}{n}}}{ \sigma_n}\right)
- \varphi_{\alpha}\left( y\right) | dy.
\end{eqnarray*}
Setting $f_n(y)=\frac{\sigma_{0,n} }{ \sigma_{n}}\varphi_{\alpha}\left( \frac{\sigma_{0,n}}{\sigma_n} y+ x\frac{1-\frac{1}{n}-e^{-\frac{1}{n}}}{ \sigma_n}\right)- \varphi_{\alpha}\left( y\right)$ and $d_n=1-\frac{1}{n}-e^{-\frac{1}{n}}$, some easy calculus give 
\begin{eqnarray*}
|f_n(y)|  = | \frac{\sigma_{0,n}}{\sigma_n}  -1 | 
\left| \varphi_{\alpha}\left( y\right)+y \frac{\sigma_{0,n} }{ \sigma_{n}}\varphi_{\alpha}^{\prime}\left( c_{y,n}\right)
+ x(\frac{\sigma_{0,n} }{ \sigma_{n}})^2  \frac{d_n}{\sigma_{0,n} ( \frac{\sigma_{0,n}}{\sigma_n}  -1  )}    \varphi_{\alpha}^{\prime}( c_{y,n}) \right|,
\end{eqnarray*}
with $c_{y,n} \in (y, \frac{\sigma_{0,n}}{\sigma_n} y+x\frac{d_n}{\sigma_n})$. Moreover, we have
$$
\frac{d_n}{\sigma_{0,n} ( \frac{\sigma_{0,n}}{\sigma_n}  -1  )}=-\frac{n^{1/ \alpha}}{n} (1 +o(1)) \mbox{ and } \frac{d_n}{\sigma_n}=-\frac{n^{1/ \alpha}}{2n^2} (1 +o(1)).
$$
This finally gives by dominated convergence
$$
d_{TV}(X_{1/n},\tilde{X}_{1/n}) \geq
\begin{cases}
\frac{C}{n} \int_0^1 |\varphi_{\alpha}\left( y\right)+y \varphi_{\alpha}^{\prime}\left(y\right) | dy \; \mbox{ if } \alpha >1 \\
\frac{C}{n} \int_0^1 |\varphi_{\alpha}\left( y\right)+y \varphi_{\alpha}^{\prime}\left(y\right)-x \varphi_{\alpha}^{\prime}\left(y\right) | dy \; \mbox{ if } \alpha =1 \\
\frac{C}{n^{2-1/ \alpha}} |x|  \int_0^1 | \varphi_{\alpha}^{\prime}\left(y\right) | dy \; \mbox{ if } \frac{1}{2} < \alpha < 1 .
\end{cases}
$$

\end{proof}
\quad

\noindent
{\it Stochastic exponential.}
We now consider the process $(X_t)_{t \geq 0}$ that solves
\begin{equation} \label{E:SE}
X_t=1+  \int_0^t X_{s-} d S_s^{\alpha, \tau}, 
\end{equation}
where $(S_t^{\alpha,\tau})_{t \geq 0}$ is a truncated stable process with L\'evy measure given by
$$
F(dz)=\frac{c_0}{|z|^{\alpha+1} }1_{\{ |z| \leq 1/2\} } dz,
$$
and admitting the representation $S_t^{\alpha,\tau}= \int_0^t \int_{\R} z \tilde{\mu}(ds,dz)$.
Since the equation has no drift, we only consider the Euler scheme and we obtain the following result.
\begin{prop}
For the stochastic exponential \eqref{E:SE}, we have for $\alpha \in (1,2)$ and for $n$ large enough
$$
d_{TV}(X_{1/n},\tilde{X}_{1/n}) \geq \frac{C(\alpha)}{n^{1/ \alpha}  (\log n)^{2/ \alpha}}, \quad C(\alpha)>0.
$$
\end{prop}

\begin{proof}
We have $\tilde{X}_{1/n}= 1+ S_{1/n}^{\alpha,\tau}$ and from It\^{o}'s formula, for $t \geq 0$,  $X_t=e^{Y_t}$ (see \cite{DA}) where
\begin{eqnarray*}
Y_t=\int_0^t \int_{\R} \log(1+z) \tilde{\mu}(ds,dz) + \int_0^t \int_{\R} ( \log(1+z)-z) F(dz) ds \\
 = S_t^{\alpha,\tau}+ \int_0^t \int_{0<|z| \leq 1/2} (\log(1+z)-z) \mu(ds,dz).
\end{eqnarray*}
Observing that for $0 < |z| \leq 1/2$, we have $z-\log(1+z) \geq z^2/4$, we deduce
\begin{eqnarray*}
d_{TV}(X_{1/n},\tilde{X}_{1/n}) &  \geq &  | \PP(X_{1/n} \geq 1)- \PP(\tilde{X}_{1/n} \geq 1) | \\
&  = & | \PP(Y_{1/n} \geq 0)- \PP(S_{1/n}^{\alpha,\tau} \geq 0)|  \\
&   \geq & \PP\left( 0 \leq  S_{1/n}^{\alpha,\tau} \leq  \frac{1}{4}\int_0^t \int_{0<|z| \leq 1/2} z^2 \mu(ds,dz)\right). 
\end{eqnarray*}
Now, for $\varepsilon_n >0$, we consider the event  $A_n=\{ \mu([0,1/n] \times \{ \varepsilon_n \leq |z| \leq 1/2 \}) \geq 1\}$. 
 We remark that on $A_n$, $\int_0^t \int_{0<|z| \leq 1/2} z^2 \mu(ds,dz) \geq \varepsilon_n^2$, this yields
$$
d_{TV}(X_{1/n},\tilde{X}_{1/n}) \geq \PP \left( \{ 0 \leq  S_{1/n}^{\alpha,\tau} \leq  \frac{\varepsilon_n^2}{4}\} \cap A_n\right) \geq \PP \left(  0 \leq  S_{1/n}^{\alpha,\tau} \leq  \frac{\varepsilon_n^2}{4} \right) -\PP(A_n^c).
$$
As done previously, we consider on the same probability space the stable process $S^{\alpha}$ and the truncated stable process $S^{\alpha, \tau}$ such that $S_{1/n}^{\alpha, \tau}=S_{1/n}^{\alpha}$ on $\Omega_n$ with $\PP(\Omega_n^c)=C(\alpha)/n$, then we deduce using that $S_{1/n}^{\alpha}$ has the  distribution of $\frac{1}{n^{1/ \alpha}} S_1^{\alpha}$
$$
d_{TV}(X_{1/n},\tilde{X}_{1/n}) \geq  \PP \left(  0 \leq  S_{1}^{\alpha} \leq  \frac{n^{1/ \alpha}\varepsilon_n^2}{4}\right) - \PP(A_n^c)- 
\frac{C(\alpha)}{n}.
$$
Since $\mu([0,1/n] \times \{ \varepsilon_n \leq |z| \leq 1/2 \})$ has a Poisson distribution with parameter
$$
\lambda_n =  \frac{2 c_0}{\alpha n} (\frac{1}{\varepsilon_n^{\alpha}} -2^{\alpha}),
$$
we have 
$
\PP(A_n^c) =e^{-\lambda_n}=e^{- \frac{2 c_0}{\alpha n \varepsilon_n^{\alpha}}} e^{\frac{C(\alpha)}{ n}}.
$
Choosing $\varepsilon_n= (2c_0/(\alpha n \log(n))^{1/ \alpha}$, we finally obtain
$$
d_{TV}(X_{1/n},\tilde{X}_{1/n}) \geq  \PP \left(  0 \leq  S_{1}^{\alpha} \leq  \frac{C(\alpha)}{(n (\log n)^2)^{1/ \alpha}}\right) - \frac{1}{n} [e^{\frac{C(\alpha)}{ n}} +C(\alpha)].
$$
Since the density of $S_1^{\alpha}$  is continuous and strictly positive (see \cite{Sato}),
 we deduce for $n$ large enough if  $\alpha >1$  that $
d_{TV}(X_{1/n},\tilde{X}_{1/n}) \geq \frac{C(\alpha)}{(n (\log n)^2)^{1/ \alpha}}$ with $C(\alpha)>0$.
\end{proof}

\begin{rem} \label{R:brownien}
If we replace the truncated stable process in \eqref{E:SE} by a Brownian motion $(B_t)_{t \geq 0}$, we have $X_{1/n}= e^{B_{1/n}-\frac{1}{2n}}$ and 
$\tilde{X}_{1/n}=1+B_{1/n}$. Consequently we deduce immediately that
\begin{eqnarray*}
d_{TV}(X_{1/n},\tilde{X}_{1/n})  \geq | \PP(X_{1/n} \geq 1)- \PP(\tilde{X}_{1/n} \geq 1) | \\
  =\PP(0 \leq B_1 \leq \frac{1}{2 \sqrt{n}}) 
  \geq \frac{C}{\sqrt{n}}.
\end{eqnarray*}
\end{rem}

\section{Local Hellinger distance and Malliavin calculus} \label{S:Malliavin}

This section is devoted to the proof of Theorems \ref{T:BHellHR0} and \ref{T:BHellG}.
 Our methodology consists in writing
the Hellinger distance as the expectation of a Malliavin weight and to control this weight. We define  Malliavin calculus with respect to the truncated L\'evy process $(L^K_t)$ specified in Section \ref{S:Not}, recalling that if {\bf A1} holds the additional truncation is useless.

\subsection{Interpolation and rescaling}
The first step consists in introducing a rescaled interpolation between the processes $(X^K_t)_{0 \leq t \leq 1/n}$ and $(\overline{X}^K_t)_{0 \leq t \leq 1/n}$ (or $(\tilde{X}^K_t)_{0 \leq t \leq 1/n}$) starting from $x$, defined in Section \ref{S:Not}.

Let us define $Y^{K,n,r}$ for $0 \leq r \leq 1$ and $0 \leq t \leq 1$ by
\begin{eqnarray} \label{E:Y}
Y^{K,n,r}_t & = & x +\frac{1}{n} \int_0^t (r b(Y_s^{K,n,r})+(1-r)b(\xi_s^{n}(x)))ds   \\
 & & + \frac{1}{n^{1/\alpha} }\int_0^t (r a(Y_{s-}^{K,n,r})+(1-r)a(x)) d L_s^{K,n} \nonumber
\end{eqnarray}
with 
\begin{equation} \label{E:REDO}
\xi_t^{n}(x)=x+ \frac{1}{n} \int_0^tb( \xi_s^{n}(x)) ds,
\end{equation}
 and where $(L_t^{K,n})_{t \in [0,1]} $ is a L\'evy process admitting the decomposition (using the symmetry of the L\'evy measure)
 \begin{equation} \label{E:Ln}
L_t^{K,n}= \int_0^t \int_{\mathbb{R}} z \tilde{\mu}^{K,n}(ds, dz), \quad t \in [0,1],
\end{equation}
where $\tilde{\mu}^{K,n}$ is a compensated Poisson random measure, $\tilde{\mu}^{K,n}= \mu^{K,n}-\overline{\mu}^{K,n}$,  with compensator $\overline{\mu}^{K,n}(dt,dz)= dt \frac{g(z/n^{1/ \alpha})}{\abs{z}^{\alpha+1}} \tau_K(z/n^{1/ \alpha})1_{\mathbb{R} \setminus \{0\}}(z) dz$. 


By construction, the process $(L_t^{K,n})_{t \in [0,1]}$ is equal in law to the rescaled truncated process $(n^{1/ \alpha} L_{t/n }^{K})_{t \in [0,1]}$. Moreover
 if $r=0$, $Y_1^{K,n,0}$ has the distribution of $\overline{X}^K_{1/n}$ starting from $x$, and if $r=1$, $Y_1^{K,n,1}$ has the distribution of $X^K_{1/n}$ starting from $x$, so we have $H_x(p^K_{1/n}, \overline{p}^K_{1/n})=H_x(Y_1^{K,n,1}, Y_1^{K,n,0})$.
 
For the Euler scheme,  to study  the Hellinger distance $H_x(p^K_{1/n}, \tilde{p}^K_{1/n})$, we proceed as previously, replacing the interpolation $Y^{K,n,r}$ by $\tilde{Y}^{K,n,r}$ with
\begin{eqnarray} 
\tilde{Y}^{K,n,r}_t  & =  & x +\frac{1}{n} \int_0^t [r b(\tilde{Y}_s^{K,n,r})+(1-r)b(x)]ds  \label{E:YHR0tilde}\\  
& & + \frac{1}{n^{1/\alpha} }\int_0^t (r a(Y_{s-}^{K,n,r})+(1-r)a(x)) d L_s^{K,n}.  \nonumber
\end{eqnarray}  
 We check easily that $\tilde{Y}^{K,n,1}_1$ has the distribution of $X^K_{1/n}$ starting from $x$ and $\tilde{Y}^{K,n,0}_1$ the distribution of $\tilde{X}^K_{1/n}$ starting from $x$.
 
 To simplify the notation, we set
\begin{align} \label{E:coeff}
b(r,y,t)= r b(y)+(1-r)b(\xi_t^n(x)) \\
\tilde{b}(r,y)= r b(y)+(1-r)b(x) \\
a(r,y)=r a(y)+(1-r)a(x),
\end{align}
so we have
$$
d Y_t^{K,n,r}= \frac{1}{n} b(r,Y_t^{K,n,r},t) dt + \frac{1}{n^{1/ \alpha} }a(r, Y_{t-}^{K,n,r}) d L_t^{K,n},
$$
$$
d \tilde{Y}_t^{K,n,r}= \frac{1}{n} \tilde{b}(r,\tilde{Y}_t^{K,n,r}) dt + \frac{1}{n^{1/ \alpha} }a(r, \tilde{Y}_{t-}^{K,n,r}) d L_t^{K,n}.
$$
Note that $\forall r \in [0,1]$, $\forall y$, $a(r,y) \geq \underline{a} >0$.
\subsection{Integration by Part}
For the reader convenience, we recall some results on Malliavin calculus for jump processes, before stating our main results. We follow \cite{CG18} Section 4.2 and also refer to  \cite{BGJ} for a complete presentation. We will work on the Poisson space associated to the measure $\mu^{K,n}$ defining the process $(L_t^{K,n})$  assuming that $n$ is fixed. By construction, the support of $\mu^{K,n}$ is contained in $[0,1] \times E_n$, where 
\begin{equation*} \label{E:En}
E_n= \{ z \in \mathbb{R}; \abs{z} <
K n^{1/ \alpha} \}.
\end{equation*}

We recall that the measure $\mu^{K,n}$ has compensator
\begin{equation} \label{E:Compens}
\overline{\mu}^{K,n} (dt, dz)=dt \frac{g(z/n^{1/ \alpha})}{ \abs{z}^{\alpha+1}}  \tau_K( z/n^{1/ \alpha})1_{\{\mathbb{R} \setminus \{0\}\}}(z) dz := dt F_{K,n}(z) dz.
\end{equation}
 
We  define  the Malliavin operators $L$ and $\Gamma$ (we omit here the dependence in $n$ and $K$) and their basic properties (see Bichteler, Gravereaux, Jacod, \cite{BGJ} Chapter IV, sections 8-9-10). For  a test function $f : [0,1] \times \mathbb{R} \mapsto \mathbb{R}$ ($f$ is measurable, $\mathcal{C}^2$ with  respect to the second variable, with  bounded derivatives, and $f \in \cap_{p \geq 1} \mathbf{L}^p( dt F_{K,n}(z) dz) $), we set $\mu^{K,n}(f )= \int_0^1 \int_{\mathbb{R}} f(t,z) \mu^{K,n}(dt,dz)$.
As auxiliary function, we consider $\rho : \mathbb{R} \mapsto [0, \infty)$ such that $\rho$ is symmetric, two times differentiable and such that $\rho(z)=z^4$ if $ z \in [0,1/2]$ and $\rho(z)=z^2$ if $z \geq 1$. Thanks to the truncation $\tau_K$, we check that $\rho$, $\rho'$ and
$\rho \frac{F'_{K,n}}{F_{K,n}}$ belong to $\cap_{p \geq 1} \mathbf{L}^p( F_{K,n}(z) dz)$. We also observe that at this stage the truncation is useless if we have for any $p \geq 1$
$$
\int_{\R} |z|^p g(z) dz < \infty.
$$
This assumption is satisfied for the tempered stable process. But to include the stable process in our study, we need to introduce the truncation function.

With the previous notation, we define the Malliavin operator $L$, on a simple functional $\mu^{K,n}(f )$ as follows

\begin{equation*} \label{E:OpL1}
L( \mu^{K,n}(f ))= \frac{1}{2} \mu^{K,n} \left( \rho' f' + \rho \frac{F'_{K,n}}{F_{K,n}} f' + \rho f^{\prime \prime} \right),
\end{equation*}
where $f'$ and $f^{\prime \prime}$ are the derivatives with respect to the second variable. 
This definition permits to construct a linear operator on a space $D \subset  \cap_{p \geq 1}\mathbf{L}^p$ which is self-adjoint :
$$
\forall \Phi, \Psi \in D, \quad \mathbb{E}\Phi L \Psi = \mathbb{E} L \Phi \Psi .
$$
We associate to $L$, the symmetric bilinear operator $\Gamma$ :
\begin{equation*} \label{E:defOpgam}
\Gamma(\Phi, \Psi)= L(\Phi \Psi) -\Phi L \Psi -\Psi L \Phi.
\end{equation*}
If $f$ and $h$ are two test functions, we have :
\begin{equation*} \label{E:Opgam}
\Gamma( \mu^{K,n}(f ), \mu^{K,n}(h ))= \mu^{K,n} \left( \rho f' h'  \right),
\end{equation*}
The operators $L$ and $\Gamma$ satisfy the chain rule property  :
\begin{equation*} \label{E:chainruleL}
LG(\Phi)=G'(\Phi) L\Phi + \frac{1}{2} G''(\Phi) \Gamma(\Phi, \Phi),
\end{equation*}
\begin{equation*} \label{E:chainrule}
\Gamma(G(\Phi), \Psi)=G'(\Phi) \Gamma(\Phi, \Psi).
\end{equation*}
These operators permit to establish the following integration by parts formula (see \cite{BGJ} Theorem 8-10 p.103).
\begin{thm} \label{P:IPP}
Let $\Phi$ and $\Psi$ be random variables in $D$, and $f$ be a bounded function with bounded derivatives up to order two. If $\Gamma(\Phi, \Phi)$ is invertible and $\Gamma^{-1}(\Phi, \Phi) \in \cap_{p \geq 1} \mathbf{L}^p$, we have
\begin{equation} \label{E:IPP}
\mathbb{E}  f'(\Phi) \Psi  = \mathbb{E} f(\Phi) \mathcal{H}_{\Phi}(\Psi),
\end{equation}
with
\begin{equation} \label{E:defpoid}
\mathcal{H}_{\Phi}( \Psi)= \Psi \frac{\Gamma(\Phi, \Gamma(\Phi, \Phi) )  }{ \Gamma^2(\Phi, \Phi) }-2 \Psi  \frac{ L \Phi}{\Gamma(\Phi, \Phi)  } - \frac{\Gamma( \Phi,  \Psi )}{\Gamma(\Phi, \Phi)}.
\end{equation}
\end{thm}

We apply now the result of Theorem \ref{P:IPP}  to the random variable $Y_1^{K,n,r}$ observing that under {\bf A0} (or {\bf A1}) and {\bf HR}, $(Y_t^{K,n,r})_{t \in [0,1]} \in D$, $\forall r \in [0,1]$ and then  the following Malliavin operators are well defined (see Section 10 in \cite{BGJ}). 
Let us introduce some more notation.
For $0 \leq t \leq 1$, we set
\begin{align}
 \Gamma(Y_t^{K,n,r} , Y_t^{K,n,r}) =U_t^{k,n,r} \label{E:DU}\\
  L(Y_t^{K,n,r}) =\mathbb{L}_t^{K,n,r} .\label{E:DL}
\end{align}
We also introduce the derivative of $Y^{K,n,r}$ with respect to $r$, denoted by $\partial_r Y^{K,n,r}$ and solving the equation
\begin{eqnarray}
d \partial_r Y_t^{K,n,r}=  & \frac{1}{n} \partial_y b(r,Y_t^{K,n,r},t)  \partial_r Y_t^{K,n,r}dt + \frac{1}{n^{1/ \alpha} }\partial_y a(r, Y_{t-}^{K,n,r})  \partial_r Y_{t-}^{K,n,r} d L_t^{K,n} \nonumber \\
&  + \frac{1}{n} \partial_r b(r,Y_t^{K,n,r},t) dt + \frac{1}{n^{1/ \alpha} }\partial_r a(r, Y_{t-}^{K,n,r})  d L_t^{K,n}, \label{E:derY} 
\end{eqnarray}
with $\partial_r Y_0^{K,n,r}=0$ and
\begin{align*}
\partial_r b(r,y,t)= b(y)- b(\xi_t^n(x)), \quad \partial_y b(r,y,t)= rb'(y), \\
\partial_r a(r,y)= a(y)-a(x), \quad \partial_y a(r,y)=ra'(y).
\end{align*}
For the vector $V_t^{K,n,r}= (Y_t^{K,n,r}, \partial_r Y_t^{K,n,r}, U_t^{K,n,r})^T$, we denote by $W_t^{K,n,r}=(W_t^{K,n,r,(i,j)})_{1 \leq i,j \leq 3}$ the matrix $\Gamma(V_t^{K,n,r},V_t^{K,n,r})$ such that
\begin{eqnarray}
U_t^{K,n,r}&= &W_t^{K,n,r,(1,1)} \nonumber\\
\Gamma(Y_t^{K,n,r} , \partial_rY_t^{K,n,r})& =&W_t^{K,n,r,(2,1)} \label{E:DW21}\\
\Gamma(Y_t^{K,n,r} , \Gamma(Y_t^{K,n,r} , Y_t^{K,n,r}) ) & =&W_t^{K,n,r,(3,1)}. \label{E:DW31}
\end{eqnarray}

With this notation, we establish the following bound for $H^2_x(p^K_{1/n}, \overline{p}^K_{1/n})$. It is obvious that the same bound holds for
$H^2_x(p^K_{1/n}, \tilde{p}^K_{1/n})$, replacing  the process $Y^{K,n,r}$ by $\tilde{Y}^{K,n,r}$, but to shorten the presentation we only state the result for $Y^{K,n,r}$.
\begin{thm} \label{T:Mal-Hell}
We assume 
{\bf HR}, {\bf A0} or {\bf A1} and that for any $ r \in [0,1]$, $U_1^{K,n,r}$ is invertible and  $(U_1^{K,n,r})^{-1}\in \cap_{p \geq 1} \mathbf{L}^p$. Then we have
$$
H^2_x(p^K_{1/n}, \overline{p}^K_{1/n})=H^2_x(Y_1^{K,n,1}, Y_1^{K,n,0}) \leq \sup_{r \in [0,1]} \E_x \left| \mathcal{H}_{Y_1^{K,n,r}}( \partial_r Y_1^{K,n,r})\right|^2,
$$
where 
\begin{equation}
\mathcal{H}_{Y_1^{K,n,r}}( \partial_r Y_1^{K,n,r})=\frac{\partial_r Y_1^{K,n,r}}{U_1^{K,n,r}} \frac{ W_1^{K,n,r,(3,1)}}{U_1^{K,n,r}} 
- 2 \partial_r Y_1^{K,n,r} \frac{\mathbb{L}_1^{K,n,r}}{U_1^{K,n,r}}
- \frac{W_1^{K,n,r,(2,1)}}{U_1^{K,n,r}}. \label{E:PMal}
\end{equation}

\end{thm}

\begin{proof}

We first observe that under {\bf A0} or {\bf A1}, {\bf HR} and assuming  $U_1^{K,n,r}$ invertible with $(U_1^{K,n,r})^{-1}\in \cap_{p \geq 1} \mathbf{L}^p$, $\forall r \in [0,1]$, the random variable $Y_1^{K,n,r}$ (starting from $x$) admits a density for any $r \in [0,1]$. Morerover this density is differentiable with respect to $r$ (the existence and the regularity of the density can be deduced from \cite{BJ83} \cite{BGJ} or \cite{Picard}). We denote by $q^{K,n,r}$ this density and by $\partial_r q^{K,n,r}$ its derivative with respect to $r$. We have
\begin{eqnarray*}
H^2_x(p^K_{1/n},\overline{p}^K_{1/n}) & = & \int_{\R}( \sqrt{ q^{K,n,1}(y)}- \sqrt{q^{K,n,0}(y)} )^2 dy \\
&= & \frac{1}{4}\int_{\R}(\int_0^1 \frac{ \partial_r q^{K,n,r}(y)} {\sqrt{ q^{K,n,r}(y)}}dr)^2 dy \\
 & \leq &  \frac{1}{4} \int_0^1 \E_x \left( \frac{\partial_r q^{K,n,r}}{q^{K,n,r}}(Y_1^{K,n,r})\right)^2 dr.
\end{eqnarray*}
Using the integration by part formula, we obtain a representation for $\frac{\partial_r q^{K,n,r}}{q^{K,n,r}}$.
Let $f$ be a smooth function, by differentiating $ r \mapsto \E f(Y_1^{K,n,r})$,  we obtain
\begin{align*}
\int f(u) \partial_r q^{K,n,r}(u) du  &=   \E f'(Y_1^{K,n,r}) \partial_r Y_1^{K,n,r} \\
 & = \E f(Y_1^{K,n,r}) \mathcal{H}_{Y_1^{K,n,r}}( \partial_r Y_1^{K,n,r}) \\
 & = \E f(Y_1^{K,n,r}) \E[\mathcal{H}_{Y_1^{K,n,r}}( \partial_r Y_1^{K,n,r}) \vert Y_1^{K,n,r}] \\
 & = \int f(u) \E[\mathcal{H}_{Y_1^{K,n,r}}( \partial_r Y_1^{K,n,r}) \vert Y_1^{K,n,r}=u] q^{K,n,r}(u) du.
\end{align*}
This gives the representation 
$$
 \frac{\partial_r q^{K,n,r}}{q^{K,n,r}}( y)= \E_x [\mathcal{H}_{Y_1^{K,n,r}}( \partial_r Y_1^{K,n,r}) | Y_1^{K,n,r}=y],
$$
and we deduce the bound
$$
H^2_x(p^K_{1/n},\overline{p}^K_{1/n}) \leq  \sup_{r \in [0,1]} \E_x \left| \mathcal{H}_{Y_1^{K,n,r}}( \partial_r Y_1^{K,n,r})\right|^2.
$$

\end{proof}
The computation of the weight $\mathcal{H}_{Y_1^{K,n,r}}( \partial_r Y_1^{K,n,r})$ is derived in the next section.


\subsection{Computation of $U_1^{K,n,r}$,  $\mathbb{L}_1^{K,n,r}$ and $W_1^{K,n,r}$} \label{Ss:Techni}

We derive here the stochastic equations satisfied by versions of  processes $(U_t^{K,n,r})_{t \in [0,1]}$, $(\mathbb{L}_t^{K,n,r})_{t \in [0,1]}$ and $(W_t^{K,n,r})_{t \in [0,1]}$, assuming {\bf HR} and {\bf A0} or {\bf A1}.
Using the result of Theorem 10-3 in \cite{BGJ} (we omit the details), we obtain the following equations. These equations are solved in the next sections.

We first check that $(U_t^{K,n,r})$ and $(\mathbb{L}_t^{K,n,r})$ solve respectively  
\begin{eqnarray}
U_t^{K,n,r} = &  \frac{2}{n} \int_0^t \partial_y b(r ,Y_s^{K,n,r},s) U_s^{K,n,r}ds + \frac{2}{n^{1/\alpha}} \int_0^t  \int_{\R} \partial_y a(r, Y_{s-}^{K,n,r}) U_{s-}^{K,n,r}z  \tilde{\mu}^{K,n}(ds,dz) \nonumber\\
 & + \frac{1}{n^{2/ \alpha}} \int_0^t \int_{\R} ( \partial_y a(r, Y_{s-}^{K,n,r}) )^2 U_{s-}^{K,n,r} z^2  \mu^{K,n}(ds,dz) \nonumber \\
 & + \frac{1}{n^{2/ \alpha}} \int_0^t \int_{\R} a(r,Y_{s-}^{K,n,r})^2 \rho(z) \mu^{K,n}(ds,dz).  \label{E:EDSU}
\end{eqnarray}

\begin{eqnarray}
\mathbb{L}_t^{K,n,r}  = &  \frac{1}{n} \int_0^t \partial_y b(r ,Y_s^{K,n,r},s) \mathbb{L}_s^{K,n,r}ds + \frac{1}{n^{1/\alpha}} \int_0^t  \int_{\R} \partial_y a(r, Y_{s-}^{K,n,r}) \mathbb{L}_{s-}^{K,n,r}z  \tilde{\mu}^{K,n}(ds,dz) \nonumber \\
 & +  \frac{1}{2n} \int_0^t \partial^2_y b(r ,Y_s^{K,n,r},s)U_{s-}^{K,n,r} ds
 + \frac{1}{2n^{1/ \alpha}} \int_0^t \int_{\R}  \partial^2_{y} a(r, Y_{s-}^{K,n,r})  U_{s-}^{K,n,r} z  \tilde{\mu}^{K,n}(ds,dz) \nonumber\\
  &
 + \frac{1}{2n^{1/ \alpha}} \int_0^t \int_{\R} a(r,Y_{s-}^{K,n,r}) (\rho'(z) +  \rho(z) \frac{ F'_{K,n}(z)}{F_{K,n}(z)}) \mu^{K,n}(ds,dz). \label{E:EDSL}
\end{eqnarray}

We write now the equation satisfied by the vector $V_t^{K,n,r}=(Y_t^{K,n,r}, \partial_r Y_t^{K,n,r}, U_t^{K,n,r} )^T$, replacing $\mu^{K,n}(ds,dz)$ by $\tilde{\mu}^{K,n}(ds,dz) + ds F_{K,n}(z)dz$ to obtain
$$
d V_t^{K,n,r}= B^{K,n,r}(V_t^{K,n,r},t) dt + \int_{\R} A^{K,n,r}(V_{t-}^{K,n,r},z) \tilde{\mu}^{K,n}(dt,dz)
$$
with $B^{K,n,r}(. ,. , . ,t): \R^3 \mapsto \R^3$ and $A^{K,n,r}:\R^4 \mapsto \R^3$ (precised below) and  $V_0^{K,n,r}=(x,0,0)^T$.

\begin{eqnarray*}
B^{K,n,r,1}(v_1,v_2,v_3,t)& = &  \frac{1}{n} b(r,v_1,t), \\
B^{K,n,r,2}(v_1,v_2,v_3,t) & =  & \frac{1}{n} ( \partial_yb(r,v_1,t) v_2 + \partial_r b(r,v_1,t)), \\
B^{K,n,r,3}(v_1,v_2,v_3,t) & = & \frac{2}{n}  \partial_yb(r,v_1,t) v_3 + \frac{1}{n^{2/ \alpha }} (\partial_y a(r,v_1))^2 v_3 \int_{\R} z^2 F_{K,n}(z)dz \\ 
&  & +  \frac{1}{n^{2/ \alpha}} a(r,v_1)^2 \int_{\R} \rho(z) F_{K,n}(z)dz,
\end{eqnarray*}

$$
A^{K,n,r}(v_1,v_2,v_3, z)= \frac{1}{n^{1/\alpha}} 
\left(
\begin{array}{c}
 a(r,v_1) z \\
( \partial_y a(r,v_1) v_2 + \partial_r a(r,v_1)) z \\
2 \partial_y a(r,v_1) v_3 z + \frac{1}{n^{1/ \alpha}} (\partial_y a(r,v_1))^2 v_3 z^2 
+  \frac{1}{n^{1/ \alpha}} a(r,v_1)^2 \rho(z) 
\end{array}
\right).
$$
We use the notation
$$
D_v B^{K,n,r}(v,t)=\left(
\begin{array}{ccc}
\partial_{v_1} B^{K,n,r,1}(v,t) & \partial_{v_2} B^{K,n,r,1}(v,t) & \partial_{v_3} B^{K,n,r,1}(v,t) \\
\partial_{v_1} B^{K,n,r,2}(v,t) & \partial_{v_2} B^{K,n,r,2}(v,t) & \partial_{v_3} B^{K,n,r,2}(v,t) \\
\partial_{v_1} B^{K,n,r,3}(v,t) & \partial_{v_2} B^{K,n,r,3}(v,t)  & \partial_{v_3} B^{K,n,r,3}(v,t)
\end{array}
\right),
$$
we obtain
$$
D_v B^{K,n,r}(v,t)=
\left(
\begin{array}{ccc}
\frac{1}{n}rb'(v_1) & 0 & 0 \\
\frac{1}{n}[r b''(v_1) v_2 +b'(v_1)] &\frac{1}{n} rb'(v_1) & 0 \\
\partial_{v_1} B^{K,n,r,3}(v,t) & 0 & \partial_{v_3} B^{K,n,r,3}(v,t) 
\end{array}
\right)
$$
with
\begin{eqnarray*}
\partial_{v_1} B^{K,n,r,3}(v,t)  = &  \frac{2}{n}  r b''(v_1) v_3 + \frac{2}{n^{2/ \alpha}} r^2 (a'a'')(v_1) v_3\int_{\R} z^2 F_{K,n}(z)dz  \\
  & +\frac{2}{n^{2/ \alpha}} r a(r,v_1) a'(v_1) \int_{\R} \rho(z) F_{K,n}(z)dz, \\
\partial_{v_3} B^{K,n,r,3}(v,t)  = & \frac{2}{n} rb'(v_1)+ \frac{1}{n^{2/ \alpha}}r^2 a'(v_1)^2 \int_{\R} z^2 F_{K,n}(z)dz.
\end{eqnarray*}
Defining analogously  the matrix $D_v A^{K,n,r}(v,z)$ and the vector $D_z A^{K,n,r}$, we have
$$
D_v A^{K,n,r}(v,z)=
\left(
\begin{array}{ccc}
\frac{1}{n^{1/\alpha}}ra'(v_1)z & 0 & 0 \\
\frac{1}{n^{1/ \alpha}}[r a''(v_1) v_2 +a'(v_1)]z &\frac{1}{n^{1/ \alpha}} ra'(v_1)z & 0 \\
\partial_{v_1} A^{K,n,r,3}(v,z) & 0 & \partial_{v_3} A^{K,n,r,3}(v,z) 
\end{array}
\right)
$$
with
$$
\begin{array}{ccl}
\partial_{v_1} A^{K,n,r,3}(v,z) & = & \frac{2}{n^{1/\alpha}}  r a''(v_1) v_3 z + \frac{2}{n^{2/ \alpha}} r^2 (a'a'')(v_1) v_3 z^2 + 
\frac{2}{n^{2/ \alpha}} r a(r,v_1) a'(v_1) \rho(z) \\
 \partial_{v_3} A^{K,n,r,3}(v,z) & = & \frac{2}{n^{1/ \alpha}}  ra'(v_1)z+ \frac{1}{n^{2/ \alpha}}r^2 a'(v_1)^2 z^2 ,
\end{array}
$$
$$
D_z A^{K,n,r}(v,z)=
\frac{1}{n^{1/ \alpha}}\left(
\begin{array}{c}
a(r,v_1) \\
ra'(v_1) v_2 + (a(v_1)-a(x)) \\
2 ra'(v_1)v_3 + \frac{2}{n^{1/ \alpha}} r^2 a'(v_1)^2 v_3 z + \frac{1}{n^{1/ \alpha}} a(r, v_1)^2 \rho'(z)
\end{array}
\right).
$$

With this notation,  the matrix $W_t^{K,n,r}$ solves 
\begin{eqnarray*} \label{E:EDSW}
W_t^{K,n,r}  =  \int_0^t [W_{s-}^{K,n,r} D_v B^{K,n,r}(V_{s-}^{K,n,r},s)^T + D_v B^{K,n,r}(V_{s-}^{K,n,r},s)(W_{s-}^{K,n,r})^T] ds \hspace{2cm}  \\
   +  \int_0^t \int_{\R} [W_{s-}^{K,n,r} D_v A^{K,n,r}(V_{s-}^{K,n,r},z)^T + D_v A^{K,n,r}(V_{s-}^{K,n,r},z)(W_{s-}^{K,n,r})^T] \tilde{\mu}^{K,n}(ds,dz) \hspace{2cm} \nonumber\\
    + \int_0^t \int_{\R} D_v A^{K,n,r}(V_{s-}^{K,n,r},z) W_{s-}^{K,n,r} D_v A^{K,n,r}(V_{s-}^{K,n,r},z)^T \mu^{K,n}(ds,dz)  
 \hspace{2cm}  \nonumber\\
    + \int_0^t \int_{\R} D_z A^{K,n,r}(V_{s-}^{K,n,r},z) D_z A^{K,n,r}(V_{s-}^{K,n,r},z)^T \rho(z) \mu^{K,n}(ds,dz).\hspace{2cm} 
    \nonumber
\end{eqnarray*}
From this, we extract directly the equations for $W^{K,n,r,(2,1)} = \Gamma(Y^{K,n,r}, \partial_r Y^{K,n,r})$  and  

$W^{K,n,r,(3,1)}= \Gamma( Y^{K,n,r}, \Gamma(Y^{K,n,r}, Y^{K,n,r}))$.

\begin{eqnarray}
W_t^{K,n,r,(2,1)} =  \frac{2}{n} \int_0^t rb'(Y_s^{K,n,r}) W_s^{K,n,r,(2,1)} ds   \hspace{4cm} \label{E:EDSW21}\\
+ \frac{2}{n^{1/ \alpha}} \int_0^t \int_{\R}ra'(Y_{s-}^{K,n,r}) W_{s-}^{K,n,r,(2,1)} z \tilde{\mu}^{K,n}(ds,dz)  \hspace{2cm} \nonumber \\
   + \frac{1}{n^{2/ \alpha}} \int_0^t \int_{\R} r^2 a'(Y_{s-}^{K,n,r})^2 W_{s-}^{K,n,r,(2,1)}z^2 \mu^{K,n}(ds,dz)\quad \quad \quad \quad \nonumber \\
   + \frac{1}{n} \int_0^t (r b''(Y_s^{K,n,r}) \partial_r Y_s^{K,n,r} + b'(Y_s^{K,n,r})) U_s^{K,n,r} ds \quad \quad \quad \quad \nonumber \\
   + \frac{1}{n^{1/\alpha}} \int_0^t \int_{\R}  (r a''(Y_{s-}^{K,n,r}) \partial_r Y_{s-}^{K,n,r} + a'(Y_{s-}^{K,n,r})) U_{s-}^{K,n,r} z\tilde{\mu}^{K,n}(ds,dz) \quad \quad \nonumber \\
    + \frac{1}{n^{2/ \alpha}} \int_0^t \int_{\R} ra'( Y_{s-}^{K,n,r}) (r a''(Y_{s-}^{K,n,r}) \partial_r Y_{s-}^{K,n,r} + a'(Y_{s-}^{K,n,r}))U_{s-}^{K,n,r}z^2 \mu^{K,n}(ds,dz) \quad \quad  \nonumber \\
     + \frac{1}{n^{2/\alpha}} \int_0^t \int_{\R} a(r,Y_{s-}^{K,n,r})( r a'(Y_{s-}^{K,n,r}) \partial_r Y_{s-}^{K,n,r} + a(Y_{s-}^{K,n,r})-a(x)) \rho(z) \mu^{K,n}(ds,dz). \nonumber
\end{eqnarray}

\begin{eqnarray}
W_t^{K,n,r,(3,1)}  =  \frac{3}{n} \int_0^t rb'(Y_s^{K,n,r}) W_s^{K,n,r,(3,1)} ds \hspace{5cm} \label{E:EDSW31}\\
 + \frac{3}{n^{1/ \alpha}} \int_0^t \int_{\R}ra'(Y_{s-}^{K,n,r}) W_{s-}^{K,n,r,(3,1)} z \tilde{\mu}^{K,n}(ds,dz) \quad \quad \nonumber \\
  + \frac{3}{n^{2/ \alpha}} \int_0^t \int_{\R} r^2 a'(Y_{s-}^{K,n,r})^2 W_{s-}^{K,n,r,(3,1)}z^2 \mu^{K,n}(ds,dz) \quad \quad  \nonumber \\
   + \frac{1}{n^{3/ \alpha}} \int_0^t \int_{\R} r^3 a'(Y_{s-}^{K,n,r})^3 W_{s-}^{K,n,r,(3,1)}z^3 \mu^{K,n}(ds,dz) \quad \quad \nonumber \\
  + \frac{2}{n} \int_0^t r b''(Y_s^{K,n,r})  (U_s^{K,n,r})^2 ds +  \frac{2}{n^{1/\alpha}} \int_0^t \int_{\R}  r a''(Y_{s-}^{K,n,r}) (U_{s-}^{K,n,r})^2z \tilde{\mu}^{K,n}(ds,dz) \quad \quad  \nonumber \\
  + \frac{2}{n^{2/ \alpha}} \int_0^t \int_{\R} [ r^2(a'a'')(Y_{s-}^{K,n,r}) U_{s-}^{K,n,r} z^2+ ra(r, Y_{s-}^{K,n,r})a'(Y_{s-}^{K,n,r})\rho(z) ]U_{s-}^{K,n,r} \mu^{K,n}(ds,dz) \quad \quad \nonumber \\
   + \frac{1}{n^{2/ \alpha}} \int_0^t \int_{\R} ra'( Y_{s-}^{K,n,r}) \left(2r a''(Y_{s-}^{K,n,r}) U_{s-}^{K,n,r}z + \frac{2}{n^{1/ \alpha}}r^2(a'a'')(Y_{s-}^{K,n,r})U_{s-}^{K,n,r}z^2  \right. \quad \quad \nonumber \\
   \left.+  \frac{2}{n^{1/ \alpha}}ra(r, Y_{s-}^{K,n,r})a'(Y_{s-}^{K,n,r})\rho(z)\right) U_{s-}^{K,n,r}z\mu^{K,n}(ds,dz) \quad \quad \nonumber \\
   + \frac{1}{n^{2/\alpha}} \int_0^t \int_{\R} a(r,Y_{s-}^{K,n,r})\left(2 r a'(Y_{s-}^{K,n,r})U_{s-}^{K,n,r} + \frac{2}{n^{1/ \alpha}}r^2 a'(Y_{s-}^{K,n,r})^2 U_{s-}^{K,n,r}z \right.  \quad \quad \nonumber \\
  \left.  +  \frac{1}{n^{1/ \alpha}}a(r,Y_{s-}^{K,n,r})^2 \rho'(z)\right) \rho(z) \mu^{K,n}(ds,dz). \hspace{3cm} \nonumber
\end{eqnarray}

\subsection{Proof of Theorem  \ref{T:BHellHR0} ($a$ constant and {\bf A0})} \label{S:ProofC}
\quad
\medskip

{\bf Part (i)}
 Assuming $a$ constant, the interpolation $Y^{K,n,r}$ between \eqref{E:KEDS} and \eqref{E:KEuler} solves the equation
\begin{eqnarray} 
Y^{K,n,r}_t  =  x +\frac{1}{n} \int_0^t [r b(Y_s^{K,n,r})+(1-r)b(\xi_s^{n}(x))]ds   
 + \frac{1}{n^{1/\alpha} } a L_t^{K,n}  \label{E:YHR0}
\end{eqnarray}
with $\xi^{n}(x)$ defined by \eqref{E:REDO} and $L^{K,n}$ by \eqref{E:Ln}.

Now, to apply Theorem \ref{T:Mal-Hell}, we check that $U_1^{K,n,r}$ is invertible and $(U_1^{K,n,r})^{-1}\in \cap_{p \geq 1} \mathbf{L}^p$.

 We start by solving the equations
 \eqref{E:derY}, \eqref{E:EDSU},  \eqref{E:EDSL},  \eqref{E:EDSW21}, \eqref{E:EDSW31} defining respectively  $\partial_r Y_1^{K,n,r}$, $U_1^{K,n,r}$, $\mathbb{L}_1^{K,n,r}$, $W_1^{K,n,r,(2,1)}$ and $W_1^{K,n,r,(3,1)}$. This is done  easily since $a$ is constant.
We define  $(Z_t^{K,n,r})_{t \in [0,1]}$ by
\begin{equation} \label{E:ZHR0}
Z_t^{K,n,r}=e^{\frac{r}{n} \int_0^t b'(Y_s^{K,n,r} )ds} .
\end{equation}
Then we obtain the following explicit expressions.
\begin{eqnarray}
\partial_r Y_1^{K,n,r} & = & \frac{Z_1^{K,n,r}}{n} \int_0^1 (Z_s^{K,n,r} )^{-1} [ b(Y_s^{K,n,r}) -b(\xi_s^{n}(x) )] ds\label{E:derY0}
\end{eqnarray}
\begin{eqnarray}
U_1^{K,n,r} & = & a^2 \frac{(Z_1^{K,n,r})^2 }{n^{2/ \alpha}} \int_0^1 \int_{\R} (Z_{s-}^{K,n,r})^{-2} \rho(z) \mu^{K,n} (ds,dz) \label{E:U0}
\end{eqnarray}
\begin{eqnarray}
\mathbb{L}_1^{K,n,r} & = &  \frac{(Z_1^{K,n,r}) }{2n} \int_0^1 (Z_s^{K,n,r} )^{-1} r b''(Y_s^{K,n,r}) U_{s-}^{K,n,r} ds  \label{E:L0}\\
 &  &    +   \frac{a Z_1^{K,n,r} }{2n^{1/ \alpha}} \int_0^1 \int_{\R} (Z_{s-}^{K,n,r} )^{-1}( \rho'(z)+ \rho(z) \frac{ F'_{K,n}(z)}{F_{K,n}(z)})  \mu^{K,n} (ds,dz) 
 \nonumber 
 \end{eqnarray}
 \begin{equation}
 W_1^{K,n,r,(2,1)}  = \frac{(Z_1^{K,n,r})^2 }{n} \int_0^1 (Z_s^{K,n,r})^{-2} U_s^{K,n,r} [rb''(Y_s^{K,n,r} )\partial_r Y_s^{K,n,r} +b'(Y_s^{K,n,r} )  ]  ds   \label{E:W20}
 \end{equation}
 \begin{eqnarray}
W_1^{K,n,r,(3,1)} & = &  \frac{2r(Z_1^{K,n,r})^3 }{n} \int_0^1(Z_s^{K,n,r})^{-3} (U_s^{K,n,r})^2 b''(Y_s^{K,n,r} )ds   \label{W30} \\
& &  + a^3\frac{(Z_1^{K,n,r})^3 }{n^{3/\alpha}} \int_0^1\int_{\R} (Z_{s-}^{K,n,r})^{-3} \rho'(z)  \rho(z) \mu^{K,n} (ds,dz). \nonumber
\end{eqnarray}
We obviously have the bounds
\begin{equation} \label{E:BZ}
\sup_{t \leq1} | Z_t^{K,n,r}| \leq C(b), \quad  \quad \sup_{t \leq1} | (Z_t^{K,n,r})^{-1}| \leq C(b).
\end{equation}
This implies that
\begin{eqnarray}
\sup_{t \leq 1} |U_t^{K,n,r} | & \leq  \frac{a^2}{n^{2/ \alpha}} C(b) \mu^{K,n}(\rho), 
\label{E:BU}\\ 
\frac{1}{ |U_1^{K,n,r} |} & \leq  C(b) \frac{n^{2/ \alpha}}{a^2 \mu^{K,n}(\rho)}. \label{E:BU-1}
\end{eqnarray}
With the definition of $\rho$, we can then check that  for any $p \geq 1$ (the constant depends on $p$ but not on $K$ and $n$)
$$
\E \left( \frac{1}{|\mu^{K,n}(\rho) |^p}\right) \leq C.
$$
The proof follows the same line as in \cite{CG18} section 4.2 equation (4.25) and we omit it.  Consequently $U_1^{K,n,r}$ is invertible and $(U_1^{K,n,r})^{-1}\in \cap_{p \geq 1} \mathbf{L}^p$. From Theorem \ref{T:Mal-Hell} it is now sufficient to bound $\E_x[\mathcal{H}_{Y_1^{K,n,r}}( \partial_r Y_1^{K,n,r})^2 ]$ where
$$
\mathcal{H}_{Y_1^{K,n,r}}( \partial_r Y_1^{K,n,r})  =\frac{\partial_r Y_1^{K,n,r}}{U_1^{K,n,r}} \frac{ W_1^{K,n,r,(3,1)}}{U_1^{K,n,r}} 
- 2 \partial_r Y_1^{K,n,r} \frac{\mathbb{L}_1^{K,n,r}}{U_1^{K,n,r}}
- \frac{W_1^{K,n,r,(2,1)}}{U_1^{K,n,r}}.
$$
We study the $L^2$-norm of each term. We first deduce from Gronwall's Lemma, 
\begin{equation} \label{E:B1}
\sup_{t \leq 1} | Y_t^{K,n,r} - \xi_t^n(x) | \leq a e^{ ||b'||_{\infty}/n} \frac{1}{n^{1/ \alpha}} \sup_{s \leq 1} |L_s^{K,n}| \leq C(a,b)\frac{1}{n^{1/ \alpha}} \sup_{s \leq 1} |L_s^{K,n}| .
\end{equation}
Combining this with \eqref{E:BZ}, \eqref{E:BU} and \eqref{E:BU-1}, we obtain the intermediate bounds
\begin{equation} \label{E:BderY}
| \partial_r Y_1^{K,n,r}| \leq \frac{C(a,b)}{n}\frac{1}{n^{1/ \alpha}} \sup_{t \in [0,1]} | L_t^{K,n} |, 
\end{equation}

\begin{equation*} \label{E:BL}
|\mathbb{L}_1^{K,n,r} | \leq \frac{C(a,b)}{n}  \frac{ \mu^{K,n}( \rho)}{  n^{2/ \alpha} } + \frac{C(a,b)}{n^{1/ \alpha}} \mu^{K,n}( |\rho' + \rho  \frac{F'_{K,n}}{F_{K,n}}  |),
\end{equation*}

\begin{equation*} \label{E:BW21}
| W_1^{K,n,r,(2,1)}| \leq \frac{C(a,b)}{n}  \frac{ \mu^{K,n}( \rho)}{  n^{2/ \alpha} } [ 1 + \frac{1}{n} \frac{\sup_{t \in [0,1]} | L_t^{K,n} |}{n^{1/ \alpha}} ],
\end{equation*}
\begin{equation*} \label{E:BW31}
|W_1^{K,n,r,(3,1)} | \leq  \frac{C(a,b)}{n}  \frac{ \mu^{K,n}( \rho)^2}{  n^{4/ \alpha} } + \frac{C(a,b)}{n^{3 / \alpha}}   \mu^{K,n}(| \rho'  \rho|).
\end{equation*}

With this background, we control each term in $\mathcal{H}_{Y_1^{K,n,r}}( \partial_r Y_1^{K,n,r})$


$$
| \frac{\partial_r Y_1^{K,n,r}}{U_1^{K,n,r}} \frac{ W_1^{K,n,r,(3,1)}}{U_1^{K,n,r}}  | \leq \frac{C(a,b)}{n} \left(\frac{\sup_{t \in [0,1]} | L_t^{K,n} |}{ n^{1+1/ \alpha}} + 
 \frac{ \sup_{t \in [0,1]} | L_t^{K,n} | \mu^{K,n}(| \rho' \rho |)}{ \mu^{K,n}(\rho)^2}\right),
$$

$$
| \partial_r Y_1^{K,n,r} \frac{\mathbb{L}_1^{K,n,r}}{U_1^{K,n,r}} | \leq  \frac{C(a,b)}{n} \left( \frac{\sup_{t \in [0,1]} | L_t^{K,n} |}{ n^{1+1/ \alpha}} + 
 \frac{ \sup_{t \in [0,1]} | L_t^{K,n} | \mu^{K,n}(| \rho' + \rho  \frac{F'_{K,n}}{F_{K,n}}  |)}{ \mu^{K,n}(\rho)} \right),
$$

$$
| \frac{W_1^{K,n,r,(2,1)}}{U_1^{K,n,r}} | \leq \frac{C(a,b)}{n} \left(1+ \frac{\sup_{t \in [0,1]} | L_t^{K,n} |}{ n^{1+1/ \alpha}}\right).
$$
This permits to deduce that
$$
| \mathcal{H}_{Y_1^{K,n,r}}( \partial_r Y_1^{K,n,r})| \leq \frac{C(a,b)}{n} (1+ \frac{1}{n} T_1 + T_2 +T_3),
$$
with
$$
T_1=\frac{\sup_{t \in [0,1]} | L_t^{K,n} |}{ n^{1/ \alpha}}, \quad
T_2= \frac{ \sup_{t \in [0,1]} | L_t^{K,n} | \mu^{K,n}(| \rho' \rho |)}{ \mu^{K,n}(\rho)^2}, 
$$
$$
 T_3=\frac{ \sup_{t \in [0,1]} | L_t^{K,n} | \mu^{K,n}(| \rho' + \rho  \frac{F'_{K,n}}{F_{K,n}}  |)}{ \mu^{K,n}(\rho)}.
$$
We first  study  the $L^2$-norm of $T_1$. 
Since $L_t^{K,n}= \int_0^t \int_{\R} z \tilde{\mu}^{K,n} (ds,dz)$, we have immediately using the definition of the compensator \eqref{E:Compens}
$$ 
\E \left| \frac{\sup_{t \in [0,1]} | L_t^{K,n} |}{ n^{1/ \alpha}}\right|^2 \leq \frac{C}{n^{2/ \alpha}} \int_0^{K n^{1/\alpha}}z^2 g(\frac{z}{n^{1/ \alpha}})\frac{1}{|z|^{\alpha+1} }dz.
$$
Since $g$ is bounded, we deduce after some calculus
\begin{equation} \label{E:T1G}
\E T_1^2=\E \left| \frac{\sup_{t \in [0,1]} | L_t^{K,n} |}{ n^{1/ \alpha}}\right|^2 \leq C(\alpha) K^{2-\alpha} /n.
\end{equation}
Now if $g$ satisfies the additional assumption $\int_{\R} |z| g(z) dz < \infty$, then
\begin{equation} \label{E:T1add}
\E T_1^2=\E \left| \frac{\sup_{t \in [0,1]} | L_t^{K,n} |}{ n^{1/ \alpha}}\right|^2 \leq C(\alpha)/n,
\end{equation}
with $C(\alpha)$ independent of $K$.

Turning to $T_2$, we decompose $L_t^{K,n}$ (using the symmetry of $F_{K,n}$)  into the small jump part and the large jump part as
$$
L_t^{K,n} = \int_0^t \int_{\{0< |z| <1\}} z \tilde{\mu}^{K,n}(ds,dz) + \int_0^t \int_{|z| \geq 1\}} z \mu^{K,n}(ds,dz).
$$
Since the small jump part is bounded in $L^p$, for any $p \geq 1$, by a constant independent of $K$, we focus on the large jump part and study the worst term  in $T_2$
$$
\frac{ \int_0^1 \int_{\R} |z| 1_{\{ |z| \geq 1\} } \mu^{K,n}(ds,dz) \mu^{K,n}(| \rho' \rho | 1_{\{ |z| \geq 1\} })}{ \mu^{K,n}(\rho 1_{\{ |z| \geq 1\} })^2}.
$$
Proceeding as in the proof of Lemma 4.3 in \cite{CG18}, we deduce that
$$
\frac{ \int_0^1 \int_{\R} |z| 1_{\{ |z| \geq 1\} } \mu^{K,n}(ds,dz) \mu^{K,n}(| \rho' \rho | 1_{\{ |z| \geq 1\} })}{ \mu^{K,n}(\rho 1_{\{ |z| \geq 1\} })^2} \leq C  \mu^{K,n}( \{ |z| \geq 1\})^{1/2}.
$$
Then observing that $ \mu^{K,n}( \{ |z| \geq 1\})$ has a Poisson distribution with parameter $\lambda_{K,n} \leq C(\alpha)$, we obtain
$$
\E T_2^2 \leq C(\alpha).
$$ 
For the last term $T_3$, the definition of $F_{K,n}$ gives for $z \neq 0$
$$
| \rho (z)\frac{F'_{K,n}(z)}{F_{K,n}(z)}| \leq C(\frac{\rho(z)}{|z|} + \frac{\rho(z)}{n^{1/ \alpha}} | \frac{g'}{g}( \frac{z}{n^{1/ \alpha}})| + \frac{\rho(z)}{n^{1/ \alpha}} | \frac{\tau_K'}{\tau_K}( \frac{z}{n^{1/ \alpha}}) |).
$$
Consequently $T_3$ can be split into three terms, $T_3 \leq T_{3,1}+T_{3,2}+T_{3,3}$ with
$$
T_{3,1}= \frac{\sup_{t \in [0,1]} | L_t^{K,n} | \mu^{K,n}(| \rho' | + |\rho/z |)}{ \mu^{K,n}(\rho)},
$$
$$
T_{3,2}=\frac{1}{n^{1/ \alpha}} \frac{ \sup_{t \in [0,1]} | L_t^{K,n} | \mu^{K,n}( \rho | \frac{g'}{g}( \frac{z}{n^{1/ \alpha}})|)}{ \mu^{K,n}(\rho)},
$$
$$
T_{3,3}=\frac{1}{n^{1/ \alpha}} \frac{\sup_{t \in [0,1]} | L_t^{K,n} | \mu^{K,n}( \rho | \frac{\tau_K'}{\tau_K}( \frac{z}{n^{1/ \alpha}})|)}{ \mu^{K,n}(\rho)}.
$$
For $T_{3,1}$, we obtain by distinguishing between the small jump part and the large jump part (as for $T_2$) 
$$
\E(T_{3,1})^2 \leq C(\alpha).
$$
Since $g'/g$ is bounded, we deduce for $T_{3,2}$
$$
\E(T_{3,2})^2 \leq C \E \left| \frac{\sup_{t \in [0,1]} | L_t^{K,n} |}{ n^{1/ \alpha}}\right|^2,
$$
and we conclude using \eqref{E:T1G} or \eqref{E:T1add}. Remark that $T_{3,2}=0$ in the stable case $g=c_0$.

Finally, considering $T_{3,3}$, we first remark that by definition of $\tau_K$
$$
T_{3,3} \leq \frac{1}{n^{1/ \alpha}} \sup_{t \in [0,1]} | L_t^{K,n} | \mu^{K,n}\left( 1_{\{K n^{1/ \alpha}/2 \leq |z| \leq K n^{1/ \alpha}  \}} |\frac{\tau_K'}{\tau_K}( \frac{z}{n^{1/ \alpha}})|\right).
$$
From Burkholder inequality (see Lemma 2.5, inequality 2.1.37 in \cite{JP}),
$$
\E \left| \frac{\sup_{t \in [0,1]} | L_t^{K,n} |}{ n^{1/ \alpha}}\right|^4 \leq C(\alpha) \frac{K^{4- \alpha}}{n},
$$
and using a change of variables and assumption \eqref{E:tau}
$$
\E \mu^{K,n}\left( 1_{\{K n^{1/ \alpha}/2 \leq |z| \leq K n^{1/ \alpha}  \}} |\frac{\tau_K'}{\tau_K}( \frac{z}{n^{1/ \alpha}})|\right)^4 \leq \frac{C(\alpha)}{n K^{4+ \alpha}}. 
$$
This permits to deduce from Cauchy Schwarz inequality
$$
\E(T_{3,3})^2 \leq C(\alpha) \frac{1}{n K^{\alpha}} \leq C(\alpha)/n.
$$

To summarize, we have established (and the worst term comes from $T_{3,2}$) 
$$
\E_x |\mathcal{H}_{Y_1^{K,n,r}}( \partial_r Y_1^{K,n,r})|^2 \leq \frac{C(a,b,\alpha)}{n^2} (1 + \frac{K^{2- \alpha}}{n}),
$$
and if we have additionally $\int_{\R} |z| g(z) dz < \infty$, then 
$$
\E_x |\mathcal{H}_{Y_1^{K,n,r}}( \partial_r Y_1^{K,n,r})|^2 \leq \frac{C(a,b, \alpha)}{n^2}.
$$
In the stable case, $T_{3,2}=0$ and the worst term is $T_1/n$
$$
\E_x |\mathcal{H}_{Y_1^{K,n,r}}( \partial_r Y_1^{K,n,r})|^2 \leq \frac{C(a,b,\alpha)}{n^2} (1 + \frac{K^{2- \alpha}}{n^3}).
$$
To simplify the presentation, we have not expressed explicitly the dependence of $C(a,b, \alpha)$ in $a$, $\alpha$ and the derivatives of $b$, but it is not difficult to check that we have
$$
C(a,b, \alpha) \leq Ce^{C ||b'||_{\infty}} ( ||b''||_{\infty}^{p_1} + a^{p_2} + \frac{1}{a^{p_3}} + \frac{1}{ \alpha^{p_4}} + \frac{1}{(2- \alpha)^{p_5}}),
$$
with $p_i \geq 1$, for $1 \leq i \leq 5$.

The proof of Theorem \ref{T:BHellHR0} (i) is finished.


\medskip

{\bf Part (ii)}

The proof follows the same lines as the one of part (i) and we only indicate the main changes observing that  \eqref{E:YHR0tilde} is obtained replacing $b(\xi_s^n(x))$  in \eqref{E:Y} by $b(x)$. 
We first deduce from Gronwall's Lemma, 
\begin{equation} \label{E:B1Euler}
\sup_{t \leq 1} | \tilde{Y}_t^{K,n,r} - x | \leq  C(a,b)( \frac{ |b(x)|}{n} +\frac{1}{n^{1/ \alpha}} \sup_{s \leq 1} |L_s^{K,n}| ).
\end{equation}
This yields
\begin{equation*} \label{E:BderYEuler}
| \partial_r \tilde{Y}_1^{K,n,r}| \leq \frac{C(a,b)}{n}( \frac{ |b(x)|}{n} +\frac{1}{n^{1/ \alpha}} \sup_{t \in [0,1]} | L_t^{K,n} |).
 \end{equation*}
 Consequently, comparing to \eqref{E:BderY}, we have the additional term $\frac{ |b(x)|}{n^2}$, so we deduce the bound  
 \begin{eqnarray*}
| \mathcal{H}_{\tilde{Y}_1^{K,n,r}}( \partial_r \tilde{Y}_1^{K,n,r})| & \leq   & \frac{C(a,b)}{n} \left(1+ \frac{1}{n} T_1 + T_2 +T_3 \right. \\
 & & \left. + \frac{|b(x)|}{n^2} + |b(x)|\frac{n^{1/ \alpha}}{n} [\frac{\mu^{K,n}(| \rho' \rho |)}{ \mu^{K,n}(\rho)^2} +
 \frac{ \mu^{K,n}(| \rho' + \rho  \frac{F'_{K,n}}{F_{K,n}}  |)}{ \mu^{K,n}(\rho)} ]\right).
 \end{eqnarray*}
We show easily that $\frac{\mu^{K,n}(| \rho' \rho |)}{ \mu^{K,n}(\rho)^2}$ and 
 $ \frac{ \mu^{K,n}(| \rho' + \rho  \frac{F'_{K,n}}{F_{K,n}}  |)}{ \mu^{K,n}(\rho)}$ are bounded in $L^2$ and  with the previous study of the terms $T_1, T_2, T_3$
we obtain the result of Theorem  \ref{T:BHellHR0} (ii).


\subsection{Proof of Theorem \ref{T:BHellG} ($a$ non constant and {\bf A1})} \label{S:ProofG}

Since $g$ is compactly supported,  $X_{1/n}$ and $\overline{X}_{1/n}$ have moments of all order and the additional truncation $\tau_{K}$ is useless ($g$ is a truncation). So from now on, the interpolation $Y^{n,r}$ and the Malliavin operators do not depend on $K$.

To solve equations \eqref{E:derY}, \eqref{E:EDSU}, \eqref{E:EDSL}, \eqref{E:EDSW21}, \eqref{E:EDSW31} (defining  $\partial_r Y_1^{n,r}$, $U_1^{n,r}$,  $\mathbb{L}_1^{n,r}$, $W_1^{n,r,(2,1)}$, $W_1^{n,r,(3,1)}$), we introduce $(Z_t^{n,r})$ that solves the linear equation
\begin{eqnarray}\label{E:Z}
Z_t^{n,r}=1 +  \frac{1}{n} \int_0^t rb'(Y_s^{n,r}) Z_s^{n,r}ds + \frac{1}{n^{1/\alpha}} \int_0^t  \int_{\R} ra'(Y_{s-}^{n,r})Z_{s-}^{n,r}z  \tilde{\mu}^n(ds,dz).
\end{eqnarray}
Under {\bf A1}, $Z_t^{n,r}$ is invertible and from It\^{o}'s formula, we check that 
\begin{eqnarray} \label{E:DY}
\quad \quad \partial_r Y_t^{n,r}  & =  & Z_t^{n,r}\int_0^t  (Z_{s-}^{n,r})^{-1} \frac{1}{n }(b(Y_s^{n,r})-b(\xi_s^{n}(x) )) ds \\
  & & +\frac{Z_t^{n,r}}{n^{1/ \alpha}}\int_0^t \int_{\R} (Z_{s-}^{n,r})^{-1}  ( a( Y_{s-}^{n,r})-a(x)) z  \tilde{\mu}^n(ds,dz) 
  \nonumber\\
 & &   - \frac{Z_t^{n,r}}{n^{1/ \alpha}}\int_0^t \int_{\R} (Z_{s-}^{n,r})^{-1} \left( \frac{  ( a(Y_{s-}^{n,r})-a(x)) }{ 1+ \frac{ ra'( Y_{s-}^{n,r}) z}{n^{1/ \alpha}} } \right) 
  \frac{ ra'(Y_{s-}^{n,r}) z^2}{n^{1/ \alpha}} \mu^n(ds,dz), \nonumber
\end{eqnarray}
\begin{eqnarray}\label{E:U}
U_t^{n,r}=\frac{(Z_t^{n,r})^2}{n^{2/ \alpha}} \int_0^t \int_{\R} (Z_{s-}^{n,r})^{-2} 
\left( \frac{ a(r, Y_{s-}^{n,r}) }{1+ \frac{ ra'( Y_{s-}^{n,r}) z}{n^{1/ \alpha}} }\right)^2 \rho(z) \mu^n(ds,dz),
\end{eqnarray}

\begin{eqnarray} \label{E:L}
 \mathbb{L}_t^{n,r}  & =  & \frac{Z_t^{n,r}}{2n} \int_0^t (Z_{s-}^{n,r})^{-1}rb''(Y_s^{n,r})U_{s-}^{n,r} ds \nonumber\\
& & +
 \frac{Z_t^{n,r}}{2n^{1/ \alpha}} \int_0^t \int_{\R} (Z_{s-}^{n,r})^{-1}
 \left[ \left( \frac{ a(r, Y_{s-}^{n,r}) }{1+ \frac{ ra'( Y_{s-}^{n,r}) z}{n^{1/ \alpha}} }\right)(\rho'(z) +  \rho(z) \frac{ F'_n(z)}{F_n(z)}) \mu^n(ds,dz)  \right. \\
 & & \left.  + ra''(Y_{s-}^{n,r})  U_{s-}^{n,r} z  \tilde{\mu}^n(ds,dz) 
  - \left( \frac{ ra''( Y_{s-}^{n,r})  U_{s-}^{n,r}  }{ 1+ \frac{ ra'( Y_{s-}^{n,r})z }{n^{1/ \alpha}} }\right)\frac{ ra'( Y_{s-}^{n,r}) z^2}{n^{1/ \alpha}} \mu^n(ds,dz) \right]. \nonumber
\end{eqnarray}
Since  equations  \eqref{E:EDSW21} and \eqref{E:EDSW31} are more complicated,  we just explicit the structure of the solution for $W_1^{n,r,(2,1)}$ and $W_1^{n,r,(3,1)}$, where $P^{n,0}$, $P^{n,1}$, $P^{n,2}$ are obtained from \eqref{E:EDSW21}  and  \eqref{E:EDSW31} respectively.
\begin{eqnarray} \label{E:W21}
 W_t^{n,r,(2,1)} & = & (Z_t^{n,r})^2 \int_0^t (Z_s^{n,r})^{-2} \left( P_s^{n,0} ds + 
\int_{\R} \frac{ P_{s-}^{n,1}(z)}{(1+ \frac{ra'(Y_{s-}^{n,r}) z}{n^{1/ \alpha}})^2} \mu^n(ds,dz) \right. \nonumber\\
  & & \left. + \int_{\R} P_{s-}^{n,2}(z) \tilde{\mu}^n(ds,dz) -  \int_{\R} P_{s-}^{n,2}(z)[1- \frac{1}{(1+ \frac{ra'(Y_{s-}^{n,r}) z}{n^{1/ \alpha}})^2} ]\mu^n(ds,dz) \right),  
\end{eqnarray}
\begin{eqnarray} \label{E:W31}
 W_t^{n,r,(3,1)} & = & (Z_t^{n,r})^3 \int_0^t (Z_s^{n,r})^{-3} \left( P_s^{n,0} ds + 
\int_{\R} \frac{ P_{s-}^{n,1}(z)}{(1+ \frac{ra'(Y_{s-}^{n,r}) z}{n^{1/ \alpha}})^3} \mu^n(ds,dz) \right.\nonumber \\
 & & \left. + \int_{\R} P_{s-}^{n,2}(z) \tilde{\mu}^n(ds,dz) -  \int_{\R} P_{s-}^{n,2}(z)[1- \frac{1}{(1+ \frac{ra'(Y_{s-}^{n,r}) z}{n^{1/ \alpha}})^3} ]\mu^n(ds,dz) \right).   
\end{eqnarray}
To identify the rate of convergence in the previous expressions and to simplify the study, we introduce some integrable processes $(P_t)_{t \in [0,1]}$ (we omit the dependence on $n$), whose expressions change  from line to line, but such that
$$
\forall n \geq 1, \; \forall r \in [0,1], \;  \quad \mathbb{E}_x \sup_{s \in [0,1]}| P_s |^p \leq C(a,b, \alpha) (1 + |x|^p), \quad \forall p \geq1.
$$
The constant $C(a,b, \alpha)$ is independent of $n$, $r$ and $x$ but depend on $p$. To avoid heavy notation, we omit the dependence on $p$, except in Lemma \ref{L:techni} below. 
We also use the notation 
\begin{equation} \label{E:Mart}
M_t= \int_0^t P_{s-} d L_s^n, \quad  R_t=\int_0^t \int_{\R} |z| 1_{\{ |z | > 1\}} \mu^n(ds,dz), \quad t \in [0,1].
\end{equation}
From Burkholder inequality,
$$
\E_x  \frac{ \sup_{t \in [0,1] } | M_t |^p}{n^{p / \alpha}} \leq C(a,b, \alpha) (1 + |x|^p),
$$
that is  $M_t/ n^{1/ \alpha}= P_t$. Moreover using $|z|/ n^{1/ \alpha} \leq 1/(2 || a' ||_{\infty})$, we also have $R_t/ n^{1/ \alpha}= P_t$. In the following, we distinguish between the small jump part and the large jump part of $M_t$ 
$$
M_t^{SJ}= \int_0^t \int_{\R} P_{s-} z 1_{\{ |z |  \leq 1\}}  \tilde{\mu}^n(ds,dz), \quad M_t^{LJ}= \int_0^t \int_{\R} P_{s-} z 1_{\{ |z | > 1\}} \mu^n(ds,dz),
$$
where we used the symmetry of the compensator for the second expression. We check that $M_t^{SJ}= P_t$ and that 
$|M_t^{LJ} | \leq P_t R_t$. 

We now give some relatively simple expressions or bounds for the variables $\partial_r Y_1^{n,r}$, $U_1^{n,r}$,  $\mathbb{L}_1^{n,r}$, $W_1^{n,r,(2,1)}$, $W_1^{n,r,(3,1)}$.
We first remark that from {\bf A1}, $\mu^{n}$ has support in $\{ |z| \leq n^{1/ \alpha} \frac{1}{2 || a' ||_{\infty}} \}$ and we have for any $y$ and any $z$ such that $|z| \leq n^{1/ \alpha} \frac{1}{2 || a' ||_{\infty}}$ 
$$
\frac{2}{3} \leq  \frac{1}{ | 1+ ra'(y) \frac{z}{n^{1/ \alpha} }|} \leq 2.
$$
Moreover standard arguments give $Z_t^{n,r}=P_t$ and $(Z_t^{n,r})^{-1}=P_t$.
This permits to deduce
\begin{equation} \label{E:BGU}
\forall t \in [0,1], \quad 0 \leq U_t^{n,r} \leq P_t \frac{\mu^n(\rho)}{n^{2/ \alpha}},
\end{equation}
\begin{equation} \label{E:BGUI}
0 \leq \frac{1}{U_1^{n,r}} \leq P_1 \frac{n^{2/ \alpha}}{\underline{a}^2\mu^n(\rho)}.
\end{equation}
So as in Section \ref{S:ProofC}, we  check that $1/ U_1^{n,r} \in \cap_{p \geq 1} \mathbf{L}^p$. 
We also observe that
\begin{equation} \label{E:BYx}
\forall t \in [0,1], \quad Y_t^{n,r} - x  = \frac{P_t}{n} + \frac{M_t}{n^{1/ \alpha}},
\end{equation}
and from Gronwall's inequality, we have
\begin{equation} \label{E:BYxsi}
\forall t \in [0,1], \quad |Y_t^{n,r} - \xi_t^{n}(x) | \leq C(b) \frac{\sup_{t \in [0,1]} |M_t| }{ n^{1/ \alpha}}.
\end{equation}
The next lemma summarizes our results, having in mind that we want to identify the rate of convergence of 
$ \partial_r Y_1^{n,r}W_1^{n,r,(3,1)}/(U_1^{n,r})^2$, $  \partial_r Y_1^{n,r} \mathbb{L}_1^{n,r}/ U_1^{n,r}$ and
$  W_1^{n,r,(2,1)}/ U_1^{n,r}$, where $U_1^{n,r}$ is approximately $\mu^n(\rho)/n^{2/ \alpha}$. 

\begin{lem} \label{L:Bounds}
With $R_t=\int_0^t \int_{\R} |z| 1_{\{ |z | > 1\}} \mu^n(ds,dz)$, we have the bounds
\begin{enumerate}
\item
\begin{eqnarray*} 
\sup_{t \in [0,1] }  |\partial_r Y_t^{n,r}|   & \leq  & \frac{P_1}{n^{1+1/ \alpha}}\left( 1 +R_1\right) \\
& & + \frac{P_1}{n^{2/ \alpha}}\left(1+R_1 +\int_0^1 \int_{\R} R_{s-} |z| 1_{\{ |z| > 1\}} \mu^n(ds,dz)  \right), \nonumber
\end{eqnarray*}
\item
\begin{eqnarray*} 
 |\mathbb{L}_1^{n,r} |  \leq \frac{P_1}{n^{1+2/ \alpha}} \mu^n(\rho) + \frac{P_1}{n^{2/ \alpha}} (1+\mu^n(\rho))
+ \frac{P_1}{n^{1/ \alpha}}  \mu^n(|\rho' + \rho \frac{F'_n}{F_n}|),
\end{eqnarray*}
\item
\begin{eqnarray*}
 |W_1^{n,r,(2,1)}|  \leq \frac{P_1}{n^{1+2/ \alpha}} \mu^n(\rho) +  \frac{P_1}{n^{2/ \alpha}} \mu^n(\rho) \sup_t |\partial_r Y_{t}^{n,r}  |
 \hspace{2cm}\\
+ \frac{P_1}{n^{3/ \alpha}}[ \mu^n(\rho) + R_1+ R_1\int_0^1 \int_{\R} R_{s-} |z | 1_{\{ |z| > 1\}} \mu^n(ds,dz)],
\end{eqnarray*}
\item
\begin{eqnarray*} 
 |W_1^{n,r,(3,1)}|  \leq  \frac{P_1}{n^{1+4/ \alpha}} \mu^n(\rho)^2  +  \frac{P_1}{n^{4/ \alpha}}(1+ \mu^n(\rho)^2 ) + \frac{P_1}{n^{3/ \alpha}} \mu^n(|\rho' \rho|).
\end{eqnarray*}
\end{enumerate}
\end{lem}
 
 \begin{proof}
1. Using equation \eqref{E:DY} with \eqref{E:BYx} and \eqref{E:BYxsi}, $\forall t \in [0,1]$
\begin{eqnarray*}
 |\partial_r Y_t^{n,r}|   \leq  \frac{P_t}{n} \frac{ \sup_t | M_t|}{n^{1/ \alpha}} + \frac{P_t}{n} \frac{1}{n^{1/ \alpha}} \int_0^t \int_{\R} \frac{z^2}{n^{1/ \alpha}} \mu^n(ds,dz) \hspace{2cm}\\
+  \frac{P_t}{n^{2/ \alpha}} | \int_0^t \int_{\R} P_{s-} M_{s-} z \tilde{\mu}^n(ds,dz) |+
 \frac{P_t}{n^{2/ \alpha}} | \int_0^t \int_{\R} P_{s-} M_{s-} \frac{z^2}{n^{1/ \alpha}} \mu^n(ds,dz)|. 
\end{eqnarray*}
In this expression to identify a sharp rate of convergence,  we distinguish between the small jumps and the large jumps for each integral. Remarking that $|z| / n^{1/ \alpha}$ is bounded,  the first two terms on the right-hand side of the inequality are bounded by
$$
\frac{P_t}{n^{1+1/ \alpha}}( 1 +R_1).
$$
Moreover the last term satisfies
$$
\frac{P_t}{n^{2/ \alpha}} | \int_0^t \int_{\R} P_{s-} M_{s-} \frac{z^2}{n^{1/ \alpha}} \mu^n(ds,dz)| \leq \frac{P_t}{n^{2/ \alpha}} 
(1+ R_1 +\int_0^t \int_{\R} R_{s-} |z|  1_{ \{ | z| > 1 \}}\mu^n(ds,dz) ).
$$
Considering $\int_0^t \int_{\R} P_{s-} M_{s-} z \tilde{\mu}^n(ds,dz)=\int_0^t M_{s-} dM_s$, we split into four integrals (small jumps and large jumps of $M$)
$$
\int_0^t \int_{\R} P_{s-} M_{s-} z \tilde{\mu}^n(ds,dz)= I^1_t + I^2_t + I^3_t +I^4_t
$$
with
$
I^1_t= \int_0^t M_{s-}^{SJ} dM_s^{SJ}= P_t,$
$$
|I^2_t|  =|  \int_0^t M_{s-}^{LJ} dM_s^{LJ}|   \leq P_t \int_0^t \int_{\R}R_{s-} |z| 1_{\{ |z| > 1\}} \mu^n(ds,dz),
$$
$$
 |I^3_t|  =  |\int_0^t M_{s-}^{SJ} dM_s^{LJ}|\leq P_t R_t, \quad  I^4_t   =  \int_0^t M_{s-}^{LJ} dM_s^{SJ}.
$$
 For $I^4$, observing that $[M^{SJ}, M^{LJ}]_t=0$, we deduce from  It\^{o}'s formula that  $\int_0^t M^{LJ}_{s-} d M_s^{SJ}=  M^{LJ}_{t}  M^{SJ}_{t} -\int_0^t M^{SJ}_{s-} d M_s^{LJ}$ and then
$
 | I^4_t|  \leq  P_t R_t.
$
 Putting together these inequalities, we finally deduce the first result.

2. On a similar way, using equation  \eqref{E:L} we obtain
\begin{eqnarray*} 
 |\mathbb{L}_1^{n,r} |  \leq \frac{P_1}{n^{1+2/ \alpha}} \mu^n(\rho) + \frac{P_1}{n^{1/ \alpha}}  \mu^n(|\rho' + \rho \frac{F'_n}{F_n}|) \hspace{4cm}\\
 + \frac{P_1}{n^{1/ \alpha}}\left(| \int_0^1 \int_{\R}P_{s-} U_{s-} z \tilde{\mu}^n(ds,dz) |+ |\int_0^1 \int_{\R} P_{s-} U_{s-}\frac{z^2}{n^{1/ \alpha}} \mu^n(ds,dz)|\right).
\end{eqnarray*}
We check easily
$$
 \frac{P_1}{n^{1/ \alpha}}|\int_0^1 \int_{\R} P_{s-} U_{s-}\frac{z^2}{n^{1/ \alpha}} \mu^n(ds,dz)| \leq  \frac{P_1}{n^{2/ \alpha}} \mu^n(\rho).
$$ 
To bound  $| \int_0^1 \int_{\R}P_{s-} U_{s-} z \tilde{\mu}^n(ds,dz) |$, we introduce the process $Q_t= \int_0^t P_{s-} \rho(z) \mu^n(ds,dz)$ and its decomposition 
$$
Q^{SJ}_t= \int_0^t P_{s-} \rho(z) 1_{\{ |z| \leq 1\} } \mu^n(ds,dz)=P_t,
$$
$$
|Q^{LJ}_t|= |\int_0^t P_{s-} \rho(z) 1_{\{ |z| > 1\} } \mu^n(ds,dz)| \leq P_t \mu^n(\rho).
$$
So we have $U_t= \frac{P_t }{n^{2/ \alpha}}Q_t$ and
$\int_0^1 \int_{\R}P_{s-} U_{s-} z \tilde{\mu}^n(ds,dz)=\frac{1}{n^{2/ \alpha}}\int_0^t Q_{s-} d M_s.
$
We conclude by splitting  $\int_0^t Q_{s-} d M_s$ into the small jumps and large jumps of $Q$ and $M$, with  It\^{o}'s formula for  
$\int_0^t Q^{LJ}_{s-} d M^{SJ}_s$ (as for $I^4$ in 1.), that
$$
 \frac{P_1}{n^{1/ \alpha}}| \int_0^1 \int_{\R}P_{s-} U_{s-} z \tilde{\mu}^n(ds,dz) | \leq  \frac{P_1}{n^{2/ \alpha}}(1+ \mu^n(\rho)).
$$

3. We turn to $W_1^{n,r,(2,1)}$. From
 \eqref{E:W21} and \eqref{E:EDSW21}, we have
 \begin{eqnarray*}
 |W_1^{n,r,(2,1)}|  \leq \frac{P_1}{n^{1+2/ \alpha}} \mu^n(\rho) + \frac{P_1}{n^{1/ \alpha}} |\int_0^1 \int_{\R} P_{s-} U_{s-}  \frac{z^2}{n^{1/ \alpha}} \mu^n(ds,dz) | \\
 + \frac{P_1}{n^{1/ \alpha}} |\int_0^1 \int_{\R} P_{s-} U_{s-}  z \tilde{\mu}^n(ds,dz) | \\
 +  \frac{P_1}{n^{2/ \alpha}} \int_0^1 \int_{\R}[ P_{s-} |\partial_r Y_{s-}^{n,r}  | + P_{s-} | Y_{s-}^{n,r} - x| ]\rho(z) \mu^n(ds,dz),
\end{eqnarray*}
where we also used for some terms that $\partial_r Y_{t}^{n,r}=P_t$ (this can be deduced from 1.).
We see easily that
$
\frac{P_1}{n^{1/ \alpha}} |\int_0^1 \int_{\R} P_{s-} U_{s-}  \frac{z^2}{n^{1/ \alpha}} \mu^n(ds,dz) | \leq \frac{P_1}{n^{2/ \alpha}} \mu^n(\rho),
$
but this does not permit  to control  $W_1^{n,r,(2,1)}/U_1^{n,r}$. So we write once again $U_t= \frac{P_t }{n^{2/ \alpha}}Q_t$ with $Q$ defined above.
 Using $\rho(z)=z^2$ if $|z| >1$, we have
$|Q^{LJ}_t |\leq P_t R_1 R_t$. Consequently we obtain
\begin{eqnarray*}
\frac{P_1}{n^{1/ \alpha}} |\int_0^1 \int_{\R} P_{s-} U_{s-}  \frac{z^2}{n^{1/ \alpha}} \mu^n(ds,dz) | & \leq & \frac{P_1}{n^{3/ \alpha}}
[\mu^n(\rho) +R_1  \\
& &  +R_1 \int_0^1 \int_{\R} R_{s-} |z| 1_{\{ |z| > 1\}} \mu^n(ds,dz)].
\end{eqnarray*}
The same inequality holds for $\frac{P_1}{n^{1/ \alpha}}\int_0^1 \int_{\R} P_{s-} U_{s-}  z \tilde{\mu}^n(ds,dz) =\frac{P_1}{n^{3/ \alpha}} \int_0^1 Q_{s-} d M_s $ by decomposing into the small jumps and large jumps of $Q$ and $M$, as already done previously. 
Finally, considering the last term, we have
$$
 \frac{P_1}{n^{2/ \alpha}} \int_0^1 \int_{\R} P_{s-} |\partial_r Y_{s-}^{n,r}  | \rho(z) \mu^n(ds,dz) \leq  \frac{P_1}{n^{2/ \alpha}} \mu^n(\rho) \sup_t |\partial_r Y_{t}^{n,r}  |,
$$
and from \eqref{E:BYx}
\begin{eqnarray*}
 \frac{P_1}{n^{2/ \alpha}} \int_0^1 \int_{\R} P_{s-} | Y_{s-}^{n,r} - x| \rho(z) \mu^n(ds,dz) \leq \frac{P_1}{n^{1+2/ \alpha}} \mu^n(\rho) \hspace{3cm}\\
+ \frac{P_1}{n^{3/ \alpha}}[ \mu^n(\rho) + R_1+ R_1\int_0^1 \int_{\R} R_{s-} |z | 1_{\{ |z| > 1\}} \mu^n(ds,dz)].
\end{eqnarray*}
This completes the proof of 3.

4. Using \eqref{E:W31} and \eqref{E:EDSW31}
\begin{eqnarray*} 
 |W_1^{n,r,(3,1)}|  \leq   \frac{P_1}{n^{1+4/ \alpha}} \mu^n(\rho)^2 + \frac{P_1}{n^{1/ \alpha}} |\int_0^1 \int_{\R} P_{s-} U_{s-}^2  \frac{z^2}{n^{1/ \alpha}} \mu^n(ds,dz) | \\
 + \frac{P_1}{n^{1/ \alpha}} |\int_0^1 \int_{\R} P_{s-} U_{s-}^2  z \tilde{\mu}^n(ds,dz) | \\
 +  \frac{P_1}{n^{2/ \alpha}} \int_0^1 \int_{\R} P_{s-} U_{s-}\rho(z) \mu^n(ds,dz) 
  +  \frac{P_1}{n^{3/ \alpha}} \mu^n(|\rho' \rho|).
\end{eqnarray*}
We have
$$
 \frac{P_1}{n^{2/ \alpha}} \int_0^1 \int_{\R} P_{s-} U_{s-}\rho(z) \mu^n(ds,dz) \leq  \frac{P_1}{n^{4/ \alpha}} \mu^n(\rho)^2,
$$
$$
\frac{P_1}{n^{1/ \alpha}} |\int_0^1 \int_{\R} P_{s-} U_{s-}^2  \frac{z^2}{n^{1/ \alpha}} \mu^n(ds,dz) | \leq \frac{P_1}{n^{4/ \alpha}} \mu^n(\rho)^2.
$$
Turning to the integral with respect to $\tilde{\mu}^n$,  $J=\int_0^1 \int_{\R} P_{s-} U_{s-}^2  z \tilde{\mu}^n(ds,dz)$, we have the representation (recalling that $U_t= \frac{P_t }{n^{2/ \alpha}}Q_t$)
\begin{eqnarray*}
J=\frac{1}{n^{4/ \alpha}}\int_0^1 (Q_{s-})^2 d M_s 
\end{eqnarray*}
and analyzing each term in the decomposition of $J$ between the large and small jumps of $Q$ and $M$, we obtain
$$
\frac{P_1}{n^{1/ \alpha}} |\int_0^1 \int_{\R} P_{s-} U_{s-}^2  z \tilde{\mu}^n(ds,dz) |\leq  \frac{P_1}{n^{4/ \alpha}}(1+ \mu^n(\rho)^2).
$$
The proof of lemma \ref{L:Bounds} is finished.
\end{proof}
Lemma \ref{L:Bounds} combined with \eqref{E:BGUI}  permits to obtain simple bounds for the Malliavin weight $\mathcal{H}_{Y_1^{K,n,r}}( \partial_r Y_1^{K,n,r})$ :
\begin{eqnarray} \label{E:P1}
\left| \frac{ \partial_r Y_1^{n,r} \mathbb{L}_1^{n,r}}{ U_1^{n,r}} \right|\leq P_1 |\partial_r Y_1^{n,r}| +P_1 n^{1/ \alpha}  |\partial_r Y_1^{n,r}| \frac{\mu^n(|\rho' + \rho \frac{F'_n}{F_n}|)}{\mu^n(\rho)},
\end{eqnarray}
\begin{eqnarray} \label{E:P2}
\left| \frac{ W_1^{n,r,(2,1)}}{ U_1^{n,r}} \right| \leq P_1\left( \frac{1}{n} +   \sup_t |\partial_r Y_{t}^{n,r}  |\right)
 \hspace{4cm}\\
+ \frac{P_1}{n^{1/ \alpha}}\left( 1 + \frac{R_1}{\mu^n(\rho)}+ \frac{R_1\int_0^1 \int_{\R} R_{s-} |z | 1_{\{ |z| > 1\}} \mu^n(ds,dz)}{\mu^n(\rho)}\right), \nonumber
\end{eqnarray}
\begin{eqnarray} \label{E:P3}
\left| \frac{ \partial_r Y_1^{n,r}W_1^{n,r,(3,1)}}{ (U_1^{n,r})^2} \right| \leq P_1| \partial_r Y_1^{n,r}| + P_1 n^{1/ \alpha}| \partial_r Y_1^{n,r}| \frac{\mu^n(|\rho' \rho|)}{\mu^n(\rho)^2}.
\end{eqnarray}
It remains to evaluate the $L^2$-norm of these three terms.  For this purpose, we establish an intermediate result.
\begin{lem} \label{L:techni}
We recall that $R_t=\int_0^t \int_{\R} |z| 1_{\{ |z | > 1\}} \mu^n(ds,dz)$. We have  $\forall \varepsilon >0$

(a)
$$
\E_x  \left( P_1\int_0^1 \int_{\R} R_{s-}  |z| 1_{ \{ |z| >1\}} \mu^n(ds,dz) \right)^{2} \leq 
C_{\varepsilon}(a,b, \alpha) (1 + |x|^2) \frac{n^{4/ \alpha}}{n^{2- \varepsilon}} , 
$$
(b) 
$$
\E_x \left(P_1 \frac{ R_1 \int_0^1 \int_{\R} R_{s-}  |z| 1_{ \{ |z| >1\}} \mu^n(ds,dz) }{ \mu^n(\rho) } \right)^{2} \leq 
\begin{cases}
C(a,b, \alpha) (1 + |x|^2), \; \mbox{if} \; \alpha >1,\\
 C_{\varepsilon}(a,b, \alpha) (1 + |x|^2) \frac{n^{2/ \alpha}}{n^{2- \varepsilon}} \; \mbox{if} \; \alpha \leq 1.
\end{cases}
$$
\end{lem}

\begin{proof}

We first recall that $\int_0^t \int_{\R} f(z) 1_{ \{ |z| >1\}}\mu^{n}(ds,dz) =\sum_{i=1}^{N_t} f(Z_i)$, where $(N_t)$ is a Poisson process with intensity $\lambda_n= \int_{\R} F_{n}(z) 1_{ \{ |z| >1\}}dz$ such that $\lambda_n \leq C(\alpha)$ and $(Z_i)_{i \geq 1} $ are i.i.d. variables with density $\frac{F_n(z) 1_{\{ |z|>1\} }}{\lambda_n}dz$.  

(a) We have 
$$
| P_1\int_0^1 \int_{\R} R_{s-}  |z| 1_{ \{ |z| >1\}} \mu^n(ds,dz) | \leq P_1 \sum_{i=1}^{N_1} |Z_i| \sum_{j=1}^{i-1} |Z_j| \leq P_1 \sum_{i \neq j} |Z_i | |Z_j|.
$$
So, we obtain from H\"older's inequality for any $p > 1$
$$
\E_x ( P_1\int_0^1 \int_{\R}R_{s-}  |z| 1_{ \{ |z| >1\}} \mu^n(ds,dz) )^{2} \leq C(a,b, \alpha) (1 + |x|^2)[ \E (\sum_{i \neq j} |Z_i | |Z_j|)^{2p}]^{\frac{1}{p}}.
$$
But we easily check that
$$
\E (\sum_{i \neq j} |Z_i | |Z_j|)^{2p} \leq \E(N_1^{4p}) [\E |Z_i |^{2p} ]^2,
$$
and that (the constant depends on $a$ through the truncation)
$$
\E | Z_i |^{2p} \leq C(a, \alpha) \frac{n^{2p/ \alpha}}{n}.
$$
This leads to
$$
[ \E (\sum_{i \neq j} |Z_i | |Z_j|)^{2p}]^{\frac{1}{p}} \leq C(a, \alpha)\frac{n^{4/ \alpha}}{n^{2/p}},
$$
and (a) is proved by choosing $p$ arbitrarily close to $1$ (recalling that $C(a, \alpha)$ depends on $p$).

(b) Observing that $\mu^n(\rho) \geq \mu_n(\rho 1_{ \{ |z| >1\}})$ and proceeding as in (a)
\begin{eqnarray*}
 |P_1 \frac{ R_1\int_0^1 \int_{\R} R_{s-}  |z| 1_{ \{ |z| >1\}} \mu^n(ds,dz) }{ \mu^n(\rho) } | \leq  
 P_1 \frac{\sum_{i=1}^{N_1} |Z_i|
 \sum_{i=1}^{N_1} |Z_i| \sum_{j=1}^{i-1} |Z_j| }{\sum_{i=1}^{N_1} |Z_i|^2 }.
\end{eqnarray*}
But using successively Cauchy Schwarz inequality and $| Z_i | |Z_j | \leq \frac{1}{2} ( |Z_i |^2+ | Z_j |^2)$
\begin{eqnarray*}
\left(\frac{\sum_{i=1}^{N_1} |Z_i|
 \sum_{i=1}^{N_1} |Z_i| \sum_{j=1}^{i-1} |Z_j| }{\sum_{i=1}^{N_1} |Z_i|^2 }\right)^2 & \leq & N_t \frac{(\sum_{i \neq j } | Z_i | |Z_j |)^2} {\sum_{i=1}^{N_1} |Z_i|^2} \\
  & \leq & N_t^2 \sum_{i \neq j } | Z_i | |Z_j |.
\end{eqnarray*}
Now  for any $p>1$ we have
$$
\E(\sum_{i \neq j } | Z_i | |Z_j |)^p \leq \E(N_1^{2p}) [\E( |Z_i|^p)]^2 \leq C(a, \alpha) (1+\frac{n^{p/ \alpha}}{n})^2
$$
If $\alpha>1$, choosing  $1<p< \alpha$ gives $\E(\sum_{i \neq j } | Z_i | |Z_j |)^p  \leq C(a, \alpha)$ and we obtain the first part of (b) from H\"{o}lder's inequality.

If $\alpha \leq 1$ then $\E(\sum_{i \neq j } | Z_i | |Z_j |)^p  \leq C(a, \alpha) \frac{n^{2p/ \alpha}}{n^2}$ and finally H\"{o}lder's inequality gives $\forall p>1$
$$
\E_x \left(P_1 \frac{ \int_0^1 \int_{\R} R_{s-}  |z| 1_{ \{ |z| >1\}} \mu^n(ds,dz) R_1 }{ \mu^n(\rho) } \right)^{2}
\leq  C(a,b, \alpha) (1 + |x|^2) \frac{n^{2/ \alpha}}{n^{2/p}},
$$
and we conclude by choosing $p$ arbitrarily close to $1$.
\end{proof}
From Lemma \ref{L:techni} (a) and Lemma \ref{L:Bounds}, we obtain immediately
\begin{eqnarray} \label{E:BFDY}
\E_x \sup_t |\partial_r Y_{t}^{n,r}  |^2 \leq C_{\varepsilon}(a,b,\alpha)(1+ |x|^2)( \frac{1}{n^2} +\frac{1}{n^{2/ \alpha}} +\frac{1}{n^{2- \varepsilon}}),
\end{eqnarray}
Consequently combining \eqref{E:BFDY}, \eqref{E:P2}, Lemma \ref{L:techni} (b) and observing that $R_1 / \mu^n(\rho) \leq 1$, we have
$$
\E_x\left| \frac{ W_1^{n,r,(2,1)}}{ U_1^{n,r}} \right|^2 \leq C_{\varepsilon}(a,b,\alpha)(1+ |x|^2)( \frac{1}{n^2} +\frac{1}{n^{2/ \alpha}} +\frac{1}{n^{2- \varepsilon}}).
$$
To control the $L^2$-norm of $\frac{ \partial_r Y_1^{n,r}W_1^{n,r,(3,1)}}{ (U_1^{n,r})^2}$, in view of \eqref{E:P3} and \eqref{E:BFDY} it remains to bound 
$$
n^{1/ \alpha} |\partial_r Y_1^{n,r} |\frac{\mu^n(|\rho' \rho|)}{\mu^n(\rho)^2} .
$$
We check that
$
\frac{\mu^n(|\rho' \rho|)}{\mu^n(\rho)^2} (1+ R_1) \leq P_1, 
$
and using $\mu^n(|\rho' \rho|)=\mu^n(|\rho' \rho|1_{\{ |z| \leq 1\}} ) + \mu^n(|\rho' \rho|1_{\{ |z| > 1\}} )$ with $\mu^n(|\rho' \rho|1_{\{ |z| > 1\}} ) \leq 
2 R_1 \mu^n(\rho)$, it yields
$$
 \frac{\mu^n(|\rho' \rho|)}{\mu^n(\rho)^2} \int_0^1 \int_{\R} R_{s-} |z| 1_{\{ |z| > 1\}} \mu^n(ds,dz) \leq P_1+
P_1 \frac{R_1\int_0^1 \int_{\R} R_{s-} |z| 1_{\{ |z| > 1\}} \mu^n(ds,dz)}{\mu^n(\rho)}.
$$
So from Lemma \ref{L:Bounds}  we have
\begin{eqnarray*}
n^{1/ \alpha} |\partial_r Y_1^{n,r} |\frac{\mu^n(|\rho' \rho|)}{\mu^n(\rho)^2}  & \leq  & P_1(\frac{1}{n} 
 + \frac{1}{n^{1/ \alpha}}) +  \frac{P_1}{n^{1/ \alpha}}\frac{R_1\int_0^1 \int_{\R} R_{s-} |z| 1_{\{ |z| > 1\}} \mu^n(ds,dz)}{\mu^n(\rho)}, 
\end{eqnarray*}
and consequently from \eqref{E:P3},  \eqref{E:BFDY} and Lemma \ref{L:techni} we conclude
\begin{eqnarray*}
\E_x\left| \frac{ \partial_r Y_1^{n,r}W_1^{n,r,(3,1)}}{ (U_1^{n,r})^2} \right|^2  \leq C_{\varepsilon}(a,b,\alpha)(1+ |x|^2)( \frac{1}{n^2} +\frac{1}{n^{2/ \alpha}} +\frac{1}{n^{2- \varepsilon}}).
\end{eqnarray*}
For the last term $ \frac{ \partial_r Y_1^{n,r} \mathbb{L}_1^{n,r}}{ U_1^{n,r}}$, in view of \eqref{E:P1} and \eqref{E:BFDY} it remains to study 
$$
P_1 n^{1/ \alpha}  |\partial_r Y_1^{n,r}| \frac{\mu^n(|\rho' + \rho \frac{F'_n}{F_n}|)}{\mu^n(\rho)}
$$
where $ \frac{F'_n}{F_n}(z)=\frac{1}{z} + \frac{1}{n^{1/ \alpha}} \frac{g'}{g}(\frac{z}{n^{1/ \alpha}})$. For any $p \geq 1$, we have using {\bf A1}
\begin{eqnarray*}
\E \int_0^1 \int_{\R} | \frac{g'}{g}(\frac{z}{n^{1/ \alpha}}) |^p 1_{\{ |z| > 1\}}\mu^n( ds,dz)  & = &
2 \int_{1}^{\frac{n^{1/ \alpha}}{2 || a' ||_{\infty}}} | \frac{g'}{g}(\frac{z}{n^{1/ \alpha}})|^p g(\frac{z}{n^{1/ \alpha}}) \frac{1}{z^{\alpha+1}} dz \\
& = & \frac{2}{n} \int_{1/n^{1/ \alpha}}^{\frac{1}{2 || a' ||_{\infty}}} | \frac{g'}{g}(u)|^p g(u) \frac{1}{u^{\alpha+1}} du \\
& \leq & \frac{C}{n} [\int_{1/n^{1/ \alpha}}^1  \frac{1}{u^{\alpha+1}} du +\int | \frac{g'}{g}(u)|^pg(u) du] \\
& \leq & C(\alpha).
\end{eqnarray*}
So it yields, introducing  $1_{\{ |z| \leq 1\}}$ and $1_{\{ |z| > 1\}}$
$$
\mu^n(|\rho' + \rho \frac{F'_n}{F_n}|) \leq P_1(1+ R_1).
$$
Next, Lemma \ref{L:Bounds} and the previous bound give 
$$
P_1 n^{1/ \alpha}  |\partial_r Y_1^{n,r}| \frac{\mu^n(|\rho' + \rho \frac{F'_n}{F_n}|)}{\mu^n(\rho)} \leq 
P_1(\frac{1}{n} 
 + \frac{1}{n^{1/ \alpha}}) +  \frac{P_1}{n^{1/ \alpha}}\frac{R_1\int_0^1 \int_{\R} R_{s-} |z| 1_{\{ |z| > 1\}} \mu^n(ds,dz)}{\mu^n(\rho)},
$$
and we conclude with \eqref{E:P1}, \eqref{E:BFDY} and Lemma \ref{L:techni}
\begin{eqnarray*} 
\E_x \left| \frac{ \partial_r Y_1^{n,r} \mathbb{L}_1^{n,r}}{ U_1^{n,r}} \right|^2 \leq C_{\varepsilon}(a,b,\alpha)(1+ |x|^2)( \frac{1}{n^2} +\frac{1}{n^{2/ \alpha}} +\frac{1}{n^{2- \varepsilon}}).
\end{eqnarray*}
Collecting all these results, we finally have proved, $\forall \varepsilon >0$
$$
\E_x | \mathcal{H}_{Y_1^{K,n,r}}( \partial_r Y_1^{K,n,r})|^2 \leq C_{\varepsilon}(a,b,\alpha)(1+ |x|^2)( \frac{1}{n^2} +\frac{1}{n^{2/ \alpha}} +\frac{1}{n^{2- \varepsilon}}).
$$ 
We can easily see that the constant $C_{\varepsilon}(a,b, \alpha)$ has exponential growth in $||b'||_{\infty}$ and polynomial growth in $||b''||_{\infty}$, $||a'||_{\infty}$, $||a''||_{\infty}$, $1/||a'||_{\infty}$, $b(0)$, $a(0)$, $1/\underline{a}$, $1/ \alpha$ and $1/(\alpha-2)$.

To complete the proof of Theorem \ref{T:BHellG}, we consider  the Euler approximation. The proof follows the same lines but the bound for $\partial_r \tilde{Y}_t^{n,r}$ has the additional term $b(x)/n^2$. So the first item in Lemma \ref{L:Bounds} is replaced by
\begin{eqnarray*} \label{E:BDY}
\sup_{t \in [0,1] }  |\partial_r \tilde{Y}_t^{n,r}|   & \leq  & \frac{P_1}{n^2} +\frac{P_1}{n^{1+1/ \alpha}}\left( 1 +R_1\right) \\
& & + \frac{P_1}{n^{2/ \alpha}}\left(1+R_1 +\int_0^1 \int_{\R} R_{s-} |z| 1_{\{ |z| > 1\}} \mu^n(ds,dz)  \right). \nonumber
\end{eqnarray*}
Since we have to control not only $\sup_{t  }  |\partial_r \tilde{Y}_t^{n,r}|$ but also $n^{1/ \alpha}\sup_{t  }  |\partial_r \tilde{Y}_t^{n,r}|$, we have the extra term $n^{1/ \alpha}/n^2$ and finally
$$
\E_x | \mathcal{H}_{\tilde{Y}_1^{K,n,r}}( \partial_r \tilde{Y}_1^{K,n,r})|^2 \leq C_{\varepsilon}(a,b,\alpha)(1+ |x|^2)(\frac{n^{2/ \alpha}}{n^4}+ \frac{1}{n^2} +\frac{1}{n^{2/ \alpha}} +\frac{1}{n^{2- \varepsilon}}).
$$

\bibliographystyle{plain}

\end{document}